\pdfoutput=1
\documentclass{article}

\usepackage[nonatbib, final]{neurips_2018}

\usepackage{wrapfig}
\usepackage{appendix}
\usepackage[utf8]{inputenc} % allow utf-8 input
\usepackage[T1]{fontenc}    % use 8-bit T1 fonts
\usepackage{hyperref}       % hyperlinks
\usepackage{url}            % simple URL typesetting
\usepackage{booktabs}       % professional-quality tables
\usepackage{amsfonts}       % blackboard math symbols
\usepackage{nicefrac}       % compact symbols for 1/2, etc.
\usepackage{microtype}      % microtypography
\usepackage{here}
\usepackage{citesort}
\usepackage{cleveref}
\usepackage[fleqn]{amsmath}
\usepackage{ascmac}
\usepackage{cite}
\usepackage[dvips]{graphicx}
\usepackage{psfrag}
\usepackage{amsmath}
\usepackage{amsbsy}
\usepackage{amssymb}
\usepackage{latexsym}
\usepackage{color}
\usepackage{algorithm}
\usepackage{algorithmic}

\usepackage{amsthm}
\newtheorem{proposition}{Proposition}
\newtheorem{example}{Example}
\newtheorem{definition}{Definition}
\newtheorem{theorem}{Theorem}
\newtheorem{lemma}{Lemma}

\newtheorem{assumption}{Assumption}
\newtheorem{fact}{Fact}
\newtheorem{remark}{Remark}

\newcommand{\signal}[1]{{\boldsymbol{#1}}}

\newcommand{\norm}[1]{\left\|#1\right\|}

\newcommand{\real}{{\mathbb R}}
\newcommand{\Real}{{\mathbb R}}

\newcommand{\integer}{{\mathbb Z}}
\newcommand{\refeq}[1]{(\ref{#1})}
\newcommand{\reftab}[1]{Table \ref{#1}}
\newcommand{\reffig}[1]{Figure \ref{#1}}

\newcommand{\argmin}{\operatornamewithlimits{argmin}}

\newcommand{\X}{{\mathcal{X}}}

\renewcommand{\H}{{\mathcal{H}}}

\newcommand{\U}{\mathcal{U}}

\newcommand{\T}{{\sf T}}

\newcommand{\x}{\times}

\newcommand{\ka}[1]{\kappa\left(#1\right)}

\newcommand{\inpro}[1]{\left<#1\right>}

\definecolor{darkgreen}{rgb}{0,.6,0}
\definecolor{medorange}{rgb}{0.7,0.3,0}
\definecolor{cyancyan}{rgb}{0.68, 0.92, 0.92}
\def\nn{\nonumber}

\title{Continuous-time Value Function Approximation \\in Reproducing Kernel Hilbert Spaces}

\author{
	Motoya~Ohnishi\\
	Keio Univ., KTH, RIKEN\\
	\texttt{motoya.ohnishi@riken.jp} \\
	\And
	Masahiro~Yukawa\\
	Keio Univ., RIKEN\\
	\texttt{yukawa@elec.keio.ac.jp} \\
	\AND
	Mikael~Johansson\\
	KTH\\
	\texttt{mikaelj@ee.kth.se} \\
	\And
	Masashi~Sugiyama\\
	RIKEN, Univ. Tokyo\\
	\texttt{masashi.sugiyama@riken.jp} \\
}

\begin{document}
	
	\maketitle
	
	\begin{abstract}
		Motivated by the success of reinforcement learning (RL) 
		for discrete-time tasks such as AlphaGo and Atari games, 
		there has been a recent surge of interest in using RL for continuous-time control of physical systems 
		(cf. many challenging tasks in OpenAI Gym and DeepMind Control Suite).
		Since discretization of time is susceptible to error, it is methodologically more
		desirable to handle the system dynamics directly in continuous time.
		However, very few techniques exist for continuous-time RL and they lack
		flexibility in value function approximation.
		In this paper, we propose a novel framework for model-based continuous-time value
		function approximation in reproducing kernel Hilbert spaces.
		The resulting framework is so flexible that it can accommodate any
		kind of kernel-based approach, such as Gaussian processes and kernel adaptive filters, and it allows us to handle
		uncertainties and nonstationarity without prior knowledge about
		the environment or what basis functions to employ.
		We demonstrate the validity of the presented framework through experiments.
	\end{abstract}
	
	\section{Introduction}
	\label{sec:intro}
	Reinforcement learning (RL) \cite{reinforcement,lewis2009reinforcement,sugiyama2015statistical} has been successful in a variety of applications such as AlphaGo and Atari games,
	particularly for discrete stochastic systems.
	Recently, application of RL to physical control tasks has also been gaining attention, because
	solving an optimal control problem (or the Hamilton-Jacobi-Bellman-Isaacs equation) \cite{optimalcontrol} directly is computationally prohibitive for complex nonlinear system dynamics and/or cost functions.
	
	In the physical world, states and actions are continuous, and many dynamical systems evolve in continuous time.
	OpenAI Gym \cite{openaigym} and DeepMind Control Suite \cite{controlsuite}
	offer several representative examples of such physical tasks.
	When handling continuous-time (CT) systems,
	CT formulations are methodologically desirable over the use of
	discrete-time (DT) formulations with the small time intervals, since such discretization is susceptible to errors.
	In terms of computational complexities and the ease of analysis, CT formulations are also more advantageous over DT counterparts for control-theoretic analyses such as {\it stability} and {\it forward invariance} \cite{khalil1996noninear}, which are useful for safety-critical applications.
	As we will show in this paper, our framework allows to constrain control inputs and/or states in a computationally efficient way. 
	
	One of the early examples of RL for CT systems \cite{baird1994reinforcement} pointed out that Q learning is incabable of learning in continuous time and proposed advantage updating.
	Convergence proofs were given in \cite{munos1998reinforcement} for systems described by stochastic differential equations (SDEs) \cite{oksendal2003stochastic} using a grid-based discretization of states and time.
	Stochastic differential dynamic programming and RL have also been studied in, for example, \cite{theodorou2010stochastic,pan2014probabilistic,theodorou2010reinforcement}.
	For continuous states and actions, function approximators are often employed instead of finely discretizing the state space to avoid the explosion of computational complexities.
	The work in \cite{doya2000reinforcement} presented an application of CT-RL by function approximators such as Gaussian networks with fixed number of basis functions.
	In \cite{vamvoudakis2010online}, it was mentioned that any continuously differentiable value function (VF) can be approximated by increasing the number of independent basis functions to infinity in CT scenarios,
	and a CT policy iteration was proposed.
	
	However, without resorting to the theory of reproducing kernels \cite{aronszajn50}, determining the number of basis functions and selecting the suitable basis function class cannot be performed systematically in general.
	Nonparametric learning is often desirable when no {\it a priori} knowledge about a suitable set of basis functions for learning is available.
	Kernel-based methods have many non-parametric learning algorithms, ranging from Gaussian processes (GPs) \cite{rasmussen2006gaussian} to kernel adaptive filters (KFs) \cite{liu_book10}, which can provably deal with
	uncertainties and nonstationarity.
	While DT kernel-based RL was studied in \cite{ormoneit2002kernel,xu2007kernel,taylor2009kernelized,sun2016online,nishiyama2012hilbert,grunewalder2012modelling,ohnishi2018safety}, for example, and the Gaussian process temporal difference (GPTD) algorithm was presented in \cite{engel2005reinforcement},
	no CT kernel-based RL has been proposed to our knowledge.
	Moreover, there is no unified framework in which existing kernel methods and their convergence/tracking analyses are straightforwardly applied to model-based VF approximation.
	
	In this paper, we present a novel theoretical framework of model-based CT-VF approximation
	in reproducing kernel Hilbert spaces (RKHSs) \cite{aronszajn50} for systems described by SDEs.
	The RKHS for learning is defined through one-to-one correspondence to a user-defined RKHS in which the VF being obtained is lying.
	We then obtain the associated kernel to be used for learning.
	The resulting framework renders any kind of kernel-based methods applicable in model-based CT-VF approximation, including GPs \cite{rasmussen2006gaussian} and KFs \cite{liu_book10}. 
	In addition, we propose an efficient barrier-certified policy update for CT systems, which implicitly enforces state constraints.
	Relations of our framework to the existing approaches for DT, DT stochastic (the Markov decision process (MDP)), CT, and CT stochastic systems are shown in \reftab{relationtab}.
	Our proposed framework covers model-based VF approximation working in RKHSs, including those for CT and CT stochastic systems.
	We verify the validity of the framework on the classical Mountain Car problem and a simulated inverted pendulum.
	
	%%%%%%%%%%
	\begin{table}[t]
		\caption{Relations to the existing approaches}
		\label{relationtab}
		\centering
		\begin{tabular}{lllll}
			\toprule
			& DT & DT stochastic (MDP) & CT & CT stochastic   \\
			\midrule
			Non kernel-based & (e.g. \cite{baird1995residual}) & (e.g. \cite{tsitsiklis1997analysis}) & (e.g. \cite{doya2000reinforcement}) & (e.g. \cite{munos1998reinforcement}) \\
			\textbf{Kernel-based} & \textbf{(e.g. \cite{engel2005reinforcement})} & \textbf{(e.g. \cite{grunewalder2012modelling})} & \textbf{(This work)} & \textbf{(This work)}  \\
			\bottomrule
		\end{tabular}
	\end{table}
	%%%%%%%%%
	%.
	
	\section{Problem setting}
	Throughout, $\Real$, $\integer_{\geq0}$, and $\integer_{>0}$ are
	the sets of real numbers, nonnegative integers, and strictly positive integers,
	respectively.
	We suppose that the system dynamics described by the SDE \cite{oksendal2003stochastic},
	\begin{align}
	dx=h(x(t),u(t))dt+\eta(x(t),u(t))dw, \label{system}
	\end{align}
	is known or learned, where $x(t)\in\Real^{n_x}$, $u(t)\in\U\subset\Real^{n_u}$, and $w$ are the state, control, and a Brownian motion of dimensions $n_x\in\integer_{>0}$, $n_u\in\integer_{>0}$, and $n_w\in\integer_{>0}$, respectively,
	$h:\Real^{n_x}\x\U\rightarrow\Real^{n_x}$ is the drift, and $\eta:\Real^{n_x}\x\U\rightarrow\Real^{n_x\x n_w}$ is the diffusion.
	A Brownian motion can be considered as a process noise, and is known to satisfy the Markov property \cite{oksendal2003stochastic}.
	Given a policy $\phi:\Real^{n_x}\rightarrow\U$, we define $h^{\phi}(x):=h(x,\phi(x))$ and $\eta^{\phi}(x):=\eta(x,\phi(x))$,
	and make the following two assumptions.
	\begin{assumption}
		\label{assump1}
		For any Lipschitz continuous policy $\phi$,
		both $h^{\phi}(x)$ and $\eta^{\phi}(x)$ are Lipschitz continuous, i.e., the stochastic process defined in \refeq{system}
		is an It\^o diffusion {\cite[Definition~7.1.1]{oksendal2003stochastic}}, which has a pathwise unique solution for $t\in[0,\infty)$.
	\end{assumption}
	\begin{assumption}		
		\label{assump2}
		The set $\X\subset\Real^{n_x}$ is compact with nonempty interior ${\rm int}(\X)$, and %and $\X\setminus{\rm int}(\X)$ is a compact manifold of class $C^3$.
		%Moreover, 
		${\rm int}(\X)$ is invariant under the system \refeq{system} with any Lipschitz continuous policy $\phi$, i.e.,
		\begin{align}
		P_{x}(x(t)\in{\rm int}(\X))=1,\;\forall x\in{\rm int}(\X),\;\forall t\geq 0,
		\end{align}
		where $P_{x}(x(t)\in{\rm int}(\X))$ denotes the probability that $x(t)$ lies in ${\rm int}(\X)$ when starting from $x(0)=x$.
	\end{assumption}
	Assumption \ref{assump2} implies that a solution of the system \refeq{system} stays in ${\rm int}(\X)$ with probability one.
	We refer the readers to \cite{khasminskii2011stochastic} for stochastic stability and invariance for SDEs.
	%%%%
	
	%%%%
	In this paper, we consider the immediate cost function\footnote{The cost function might be obtained by the negation of the reward function.} $R:\Real^{n_x}\x\U\rightarrow\Real$, which is continuous and satisfies
	$
	E_x\left[\int_{0}^{\infty}e^{-\beta t}|R(x(t),u(t))|dt\right]<\infty, \label{expe}
	$
	where $E_x$ is the expectation for all trajectories (time evolutions of $x(t)$) starting from $x(0)=x$, and $\beta\geq0$ is the discount factor.  Note this boundedness implies that $\beta>0$ or that there exists a zero-cost state which is stochastically asymptotically stable \cite{khasminskii2011stochastic}.
	Specifically, we consider the case where the immediate cost is {\em not} known {\em a priori} but is sequentially observed.
	Now, the VF associated with a policy $\phi$ is given by
	\begin{align}
	V^{\phi}(x):=E_x\left[\int_{0}^{\infty}e^{-\beta t}R^{\phi}(x(t))dt\right]<\infty,
	\end{align}
	where $R^{\phi}(x(t)):=R(x(t),\phi(x(t)))$.
	%%%
	
	%%%
	The advantages of using CT formulations include a smooth control performance and an efficient policy update, and CT formulations require no elaborative partitioning of time  \cite{doya2000reinforcement}.  In addition, our work shows that CT formulations make control-theoretic analyses easier and computationally more efficient
	and are more advantageous in terms of susceptibility to errors when the time interval is small.
	We mention that the CT formulation can still be considered in spite of the fact that the algorithm is implemented in discrete time.
	%%%
	
	%%%
	With these problem settings in place, our goal is to estimate the CT-VF in an RKHS and improve policies.
	However, since the output $V^{\phi}(x)$ is unobservable and the so-called double-sampling problem exists when approximating VFs (see e.g., \cite{sutton2009convergent, konda2004convergence}),
	kernel-based supervised learning and its analysis cannot be directly applied to VF approximation in general.
	Motivated by this fact, we propose a novel model-based CT-VF approximation framework which
	enables us to conduct kernel-based VF approximation as supervised learning.
	\section{Model-based CT-VF approximation in RKHSs}
	In this section, we briefly present an overview of our framework;
	We take the following steps:
	\begin{enumerate}
		\item Select an RKHS $\H_V$ which is supposed to contain $V^{\phi}$ as one of its elements.
		\item Construct another RKHS $\H_{R}$ under one-to-one correspondence to $\H_V$ through a certain bijective linear operator $U:\H_V\rightarrow\H_{R}$ to be defined later in the next section.
		\item Estimate the immediate cost function $R^{\phi}$ in the RKHS $\H_{R}$ by kernel-based supervised learning, and return its estimate $\hat{R^{\phi}}$.
		\item An estimate of the VF $V^{\phi}$ is immediately obtained by $U^{-1}(\hat{R^{\phi}})$.
	\end{enumerate}
	An illustration of our framework is depicted in \reffig{fig:illusttheo}.
	%%%%%%%%%
	\begin{figure}[t]
		\begin{center}
			\includegraphics[clip,width=0.6\textwidth]{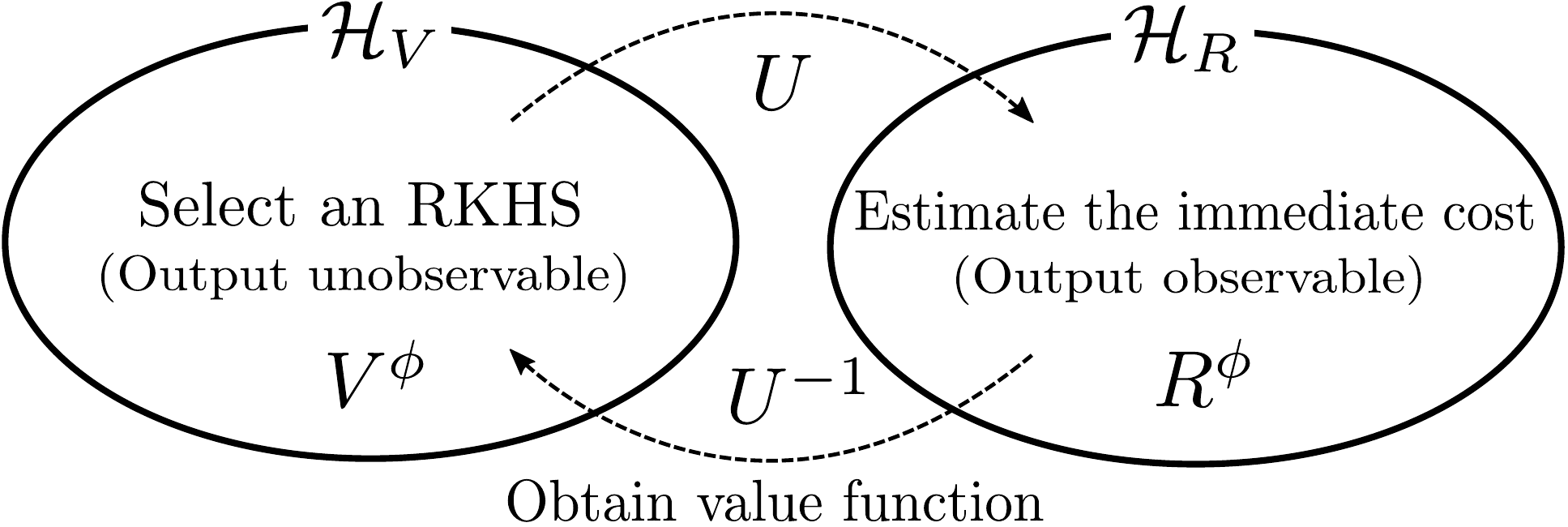}
			\caption{An illustration of the main ideas of our proposed framework.  Given a system dynamics and an RKHS $\H_V$ for the VF $V^{\phi}$, define $\H_{R}$ under
				one-to-one correspondence to estimate an observable immediate cost function in $\H_{R}$, and obtain $V^{\phi}$ by bringing it back to $\H_V$.}
			\label{fig:illusttheo}
		\end{center}
	\end{figure}
	Note we can avoid the double-sampling problem because the operator $U$ is deterministic even though the system dynamics is stochastic.
	Therefore, under this framework, model-based CT-VF approximation in RKHSs can be derived, and convergence/tracking analyses of kernel-based supervised learning can also be applied to VF approximation.
	%%%%%%%%
	\vspace{-1em}
	\paragraph {Policy update while restricting certain regions of the state space}
	\begin{algorithm}[t]
		\caption{Model-based CT-VF Approximation in RKHSs with Barrier-Certified Policy Updates}
		\label{algo}
		\begin{algorithmic}
			\STATE {\bfseries Estimate of the VF:}
			$\hat{V}^{\phi}_n=U^{-1}(\hat{R}^{\phi}_n)$
			\FOR{$n\in\integer_{\geq0}$}
			\STATE - Receive $x_n\in\X$, $\phi(x_n)\in\U$, and $R(x_n,\phi(x_n))\in\Real$
			\STATE - Update the estimate $\hat{R}^{\phi}_n$ of $R^{\phi}$ by using some kernel-based method in $\H_{R}$\hfill $\triangleright$ e.g., Section \ref{sec:applications}\\
			\STATE - Update the policy with barrier certificates when $V^{\phi}$ is well estimated\hfill $\triangleright$ e.g., \refeq{poliplus}
			\ENDFOR
		\end{algorithmic}
	\end{algorithm}
	As mentioned above, one of the advantages of a CT framework is its affinity for control-theoretic analyses such as
	\emph{stability} and \emph{forward invariance}, which are useful for safety-critical applications.
	For example, suppose that we need to restrict the region of exploration in the state space
	to some set $\mathcal{C}:=\{x\in\X\mid b(x)\geq0\}$, where $b:\X\rightarrow\Real$ is smooth.
	This is often required for safety-critical applications.
	\begin{wrapfigure}{L}{0.35\textwidth}
		\begin{center}
			\includegraphics[clip,width=0.35\textwidth]{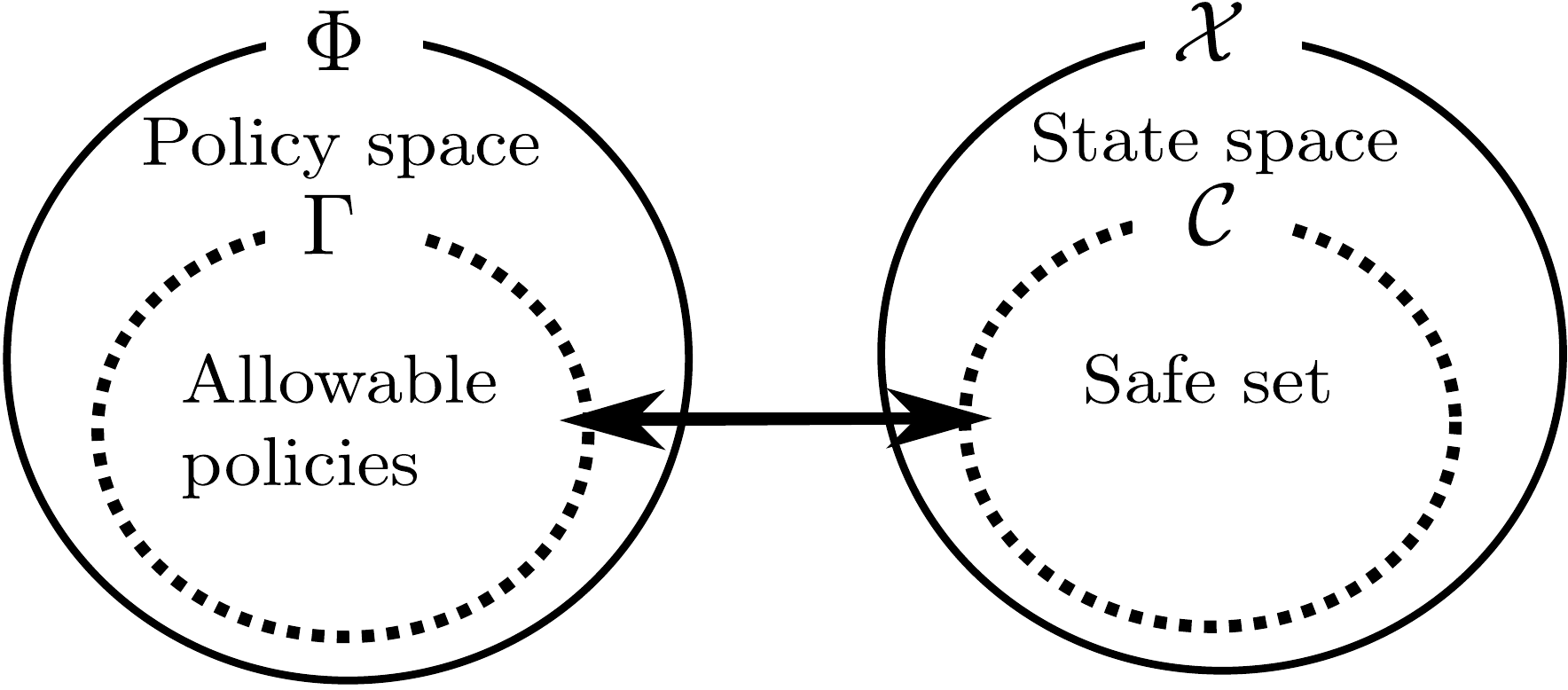}
			\caption{An illustration of barrier-certified policy updates.  State constraints are implicitly enforced via barrier certificates.}
			\label{fig:illustbarrier}
		\end{center}
	\end{wrapfigure}

	To this end, control inputs must be properly constrained so that the state trajectory remains inside the set $\mathcal{C}$.
	In the safe RL context, there exists an idea of considering a smaller space of allowable policies (see \cite{Safesurvey} and references therein).
	To effectively constrain policies, we employ control barrier certificates (cf. \cite{xu2015robustness,wieland2007constructive,glotfelter2017nonsmooth,wang2017safety,ames2016control,2017dcbf}).
	Without explicitly calculating the state trajectory over a long time horizon,
	it is known that any Lipschitz continuous policy satisfying control barrier certificates renders the set $\mathcal{C}$ forward invariant \cite{xu2015robustness}, i.e., the state trajectory remains inside the set $\mathcal{C}$.  In other words, we can implicitly enforce state constraints by satisfying barrier certificates when updating policies.
	Barrier-certified policy update was first introduced in \cite{ohnishi2018safety} for DT systems, but is computationally more efficient in
	our CT scenario.
	This concept is illustrated in \reffig{fig:illustbarrier}, where $\Phi$ is the space of Lipschitz continuous policies $\phi:\X\rightarrow\mathcal{U}$, and $\Gamma$ is the space of barrier-certified allowable policies.
	%%%%
	
	%%%%
	A brief summary of the proposed model-based CT-VF approximation in RKHSs is given in Algorithm \ref{algo}.
	In the next section, we present  theoretical analyses of our framework.
	%%%%%%

	\section{Theoretical analyses}
	We presented the motivations and an overview of our framework in the previous section.
	In this section, we validate the proposed framework from theoretical viewpoints.
	Because the output $V^{\phi}(x)$ of the VF is unobservable, we follow the strategy presented in the previous section.
	First, by properly identifying the RKHS $\H_V$ which is supposed to contain the VF, we can implicitly restrict the class of the VF.
	If the VF $V^{\phi}$ is twice continuously differentiable\footnote{
		See, for example, {\cite[Chapter~IV]{fleming2006controlled}},\cite{krylov2008controlled}, for more detailed arguments about the conditions under which twice continuous differentiability is guaranteed.} over ${\rm int}(\X)\subset\X$,
	we obtain the following Hamilton-Jacobi-Bellman-Isaacs equation \cite{oksendal2003stochastic}:
	\begin{align}
	\beta V^{\phi}(x)=-\mathcal{G}(V^{\phi})(x)+R^{\phi}(x), \;x\in{\rm int}(\X),\label{bellman}
	\end{align}
	where the infinitesimal generator $\mathcal{G}$ is defined as
	\begin{align}
	\mathcal{G}(V^{\phi})(x):=-\frac{1}{2}{\rm tr}\left[\frac{\partial^2 V^{\phi}(x)}{\partial x^2}A^{\phi}(x)\right]-
	\frac{\partial V^{\phi}(x)}{\partial x}h^{\phi}(x), \; x\in{\rm int}(\X). \label{inf}
	\end{align}
	Here, ${\rm tr}$ stands for the trace, and $A^{\phi}(x):=A(x,\phi(x))\in\Real^{n_x\x n_x}$, where $A(x,u)=\eta(x,u)\eta(x,u)^{\T}$.
	By employing a suitable RKHS such as a Gaussian RKHS for $\H_V$, we can guarantee twice continuous differentiability of an estimated VF.
	Note that functions in a Gaussian RKHS are smooth \cite{minh2010some}, and any continuous function on every compact subset of $\Real^{n_x}$ can be
	approximated with an arbitrary accuracy \cite{steinwart01} in a Gaussian RKHS.
	
	Next, we need to construct another RKHS $\H_R$ which contains the immediate cost function $R^{\phi}$ as one of its element.
	The relation between the VF and the immediate cost function is given by rewriting \refeq{bellman} as
	\begin{align}
	R^{\phi}(x)=\left[\beta I_{\rm op}+\mathcal{G}\right](V^{\phi})(x), \;x\in{\rm int}(\X),
	\end{align}
	where $I_{\rm op}$ is the identity operator.
	To define the operator $\beta I_{\rm op}+\mathcal{G}$ over the whole $\X$, we use the following definition.
	\begin{definition}[{\cite[Definition~1]{zhou2008derivative}}]
		Let $I_s:=\{\alpha:=[\alpha^1,\alpha^2,\ldots,\alpha^{n_x}]^{\T}\in\integer_{\geq0}^{n_x}\mid\sum_{j=1}^{n_x}\alpha^j\leq s\}$ for $s\in\integer_{\geq0},\;n_x\in\integer_{>0}$.
		Define
		$
		D^{\alpha}\varphi(x)=\frac{\partial^{\sum_{j=1}^{n_x}\alpha^j}}{\partial(x^1)^{\alpha^1}\partial(x^2)^{\alpha^2}\ldots\partial(x^{n_x})^{\alpha^{n_x}}}\varphi(x), %\label{dop}
		$
		where $x:=[x^1,x^2,\ldots,x^{n_x}]^{\T}\in\Real^{n_x}$.
		If $\X\subset\Real^{n_x}$ is compact with nonempty interior ${\rm int}(\X)$,
		$C^s({\rm int}(\X))$ is the space of functions $\varphi$ over ${\rm int}(\X)$ such that $D^{\alpha}\varphi$ is well defined and continuous over ${\rm int}(\X)$
		for each $\alpha\in I_s$.
		Define $C^s(\X)$ to be the space of continuous functions $\varphi$ over $\X$ such that $\varphi|_{{\rm int}(\X)}\in C^s({\rm int}(\X))$ and that $D^{\alpha}(\varphi|_{{\rm int}(\X)})$ has a continuous extension $D^{\alpha}\varphi$ to $\X$ for each $\alpha\in I_s$.
		If $\kappa\in C^{2s}(\X\x\X)$, define
		$
		(D^{\alpha}\kappa)_x(y)=\frac{\partial^{\sum_{j=1}^{n_x}\alpha^j}}{\partial(x^1)^{\alpha^1}\partial(x^2)^{\alpha^2}\ldots\partial(x^{n_x})^{\alpha^{n_x}}}
		\kappa(y,x),\;\forall x,y\in{\rm int}(\X).
		$
	\end{definition}
	Now, suppose that $\H_V$ is an RKHS associated with the reproducing kernel $\kappa^V(\cdot,\cdot)\in C^{2\x2}(\X\x\X)$.
	Then, we define the operator $U:\H_V\rightarrow\H_R:=\{\varphi\mid\varphi(x)=U(\varphi^V)(x),\;\exists\varphi^V\in\H_{V},\;\forall x\in\X\}$ as
	\begin{align}
	U(\varphi^V)(x)&:=\beta\varphi^V(x) -[D^{e_1}\varphi^V(x),D^{e_2}\varphi^V(x),\ldots,D^{e_{n_x}}\varphi^V(x)]h^{\phi}(x)\nn\\
	&-\frac{1}{2}\sum_{m,n=1}^{n_x}A^{\phi}_{m,n}(x)D^{e_m+e_n}\varphi^V(x),\;\;\;\;\; \forall \varphi^V\in\H_V,\;\forall x\in\X, \label{opeU}
	\end{align}
	where $A^{\phi}_{m,n}(x)$ is the $(m,n)$ entry of $A^{\phi}(x)$.  Note here that $U(\varphi^V)(x)=\left[\beta I_{\rm op}+\mathcal{G}\right](\varphi^V)(x)$ \emph{over ${\rm int}(\X)$}.
	We emphasize here that the {\em expected} value and the immediate cost are related through the {\em deterministic} operator $U$.
	The following main theorem states that $\H_R$ is indeed an RKHS under Assumptions \ref{assump1} and \ref{assump2}, and its corresponding reproducing kernel is obtained.
	%%%%
	%
	%%%%%%%%%
	\begin{theorem}
		\label{theoRKHS}
		Under Assumptions \ref{assump1} and \ref{assump2}, suppose that $\H_V$ is an RKHS associated with the reproducing kernel $\kappa^V(\cdot,\cdot)\in C^{2\x2}(\X\x\X)$.
		Suppose also that (i) $\beta>0$, or that (ii) $\H_V$ is a Gaussian RKHS, and there exists a point $x_{t\rightarrow\infty}\in{\rm int}(\X)$ which is stochastically asymptotically stable over ${\rm int}(\X)$, i.e., $P_x\left(\displaystyle\lim_{t\rightarrow\infty}x(t)=x_{t\rightarrow\infty}\right)=1$ for any starting point $x\in{\rm int}(\X)$.
		Then, the following statements hold.\\
		(a) The space $\H_{R}:=\{\varphi\mid\varphi(x)=U(\varphi^V)(x),\;\exists\varphi^V\in\H_{V},\;\forall x\in\X\}$ is an isomorphic Hilbert space of $\H_V$ equipped with the inner product defined by
		\begin{align}
		&\inpro{\varphi_1,\varphi_2}_{\H_{R}}:=\inpro{\varphi_1^V,\varphi_2^V}_{\H_{V}},\;\varphi_i(x):=U(\varphi_i^V)(x),\;\forall x\in\X,\;i\in\{1,2\}, \label{inpro}
		\end{align}
		where the operator $U$ is defined in \refeq{opeU}.\\
		(b) The Hilbert space $\H_{R}$ has the reproducing kernel given by
		\begin{align}
		&\kappa(x,y):=U(K(\cdot,y))(x), \;x,y\in\X,  \label{kernelCT}
		\end{align}
		where
		\begin{align}
		K(x,y)&=\beta\kappa^V(x,y)-[(D^{e_1}\kappa^V)_y(x),(D^{e_2}\kappa^V)_y(x),\ldots,(D^{e_{n_x}}\kappa^V)_y(x)]h^{\phi}(y)\nn\\
		&-\frac{1}{2}\sum_{m,n=1}^{n_x}A^{\phi}_{m,n}(y)(D^{e_m+e_n}\kappa^V)_y(x). \label{KK}
		\end{align}
	\end{theorem}	
	\vspace{-1em}
	\begin{proof}	
		See Appendices A and B in the supplementary document.
	\end{proof}	
	\vspace{-1em}
	Under Assumptions \ref{assump1} and \ref{assump2}, Theorem \ref{theoRKHS} implies that the VF $V^{\phi}$ can be uniquely determined by the immediate cost function $R^{\phi}$ for a policy $\phi$ if the VF is in an RKHS of a particular class.
	In fact, the relation between the VF and the immediate cost function in \refeq{bellman} is based on the assumption that the VF is twice continuously differentiable over ${\rm int}(\X)$, and the verification theorem (cf. \cite{fleming2006controlled}) states that, when the immediate cost function and a twice continuously differentiable function satisfying the relation \refeq{bellman} are given under certain conditions, the twice continuously differentiable function is indeed the VF.  In Theorem \ref{theoRKHS}, on the other hand, we first restrict the class of the VF by identifying an RKHS $\H_V$, and then approximate the immediate cost function in the RKHS $\H_R$ any element of which satisfies the relation \refeq{bellman}.
	Because the immediate cost $R^{\phi}(x(t))$ is observable, we can employ any kernel-based supervised learning to estimate the function $R^{\phi}$ in the RKHS $\H_{R}$, such as GPs and KFs, as elaborated later in Section \ref{sec:applications}.
	
	In the RKHS $\H_R$, an estimate of $R^{\phi}$ at time instant $n\in\integer_{\geq0}$ is given by
	$\hat{R}^{\phi}_n(x)=\sum_{i}^{r}c_i\kappa(x,x_i),\;c_i\in\Real,\;r\in\integer_{\geq0}$,
	where $\{x_i\}_{i\in\{1,2,...,r\}}\subset\X$ is the set of samples, and the reproducing kernel $\kappa$ is defined in \refeq{kernelCT}.
	An estimate of the VF $V^{\phi}$ at time instant $n\in\integer_{\geq0}$ is thus immediately obtained by
	$\hat{V}^{\phi}_n(x)=U^{-1}(\hat{R}^{\phi}_n)(x)=\sum_{i=1}^{r}c_iK(x,x_i)$, where $K$ is defined in \refeq{KK}.
	
	%%%%%%
	Note, when the system dynamics is described by an ordinary differential equation (i.e., $\eta=0$), the assumptions that $V^{\phi}$ is twice continuously differentiable and that
	$\kappa^V(\cdot,\cdot)\in C^{2\x2}(\X\x\X)$ are relaxed to that $V^{\phi}$ is continuously differentiable and that $\kappa^V(\cdot,\cdot)\in C^{2\x1}(\X\x\X)$, respectively.
	%%%%%%%%%%%%
	
	%%%%%%
	As an illustrative example of Theorem \ref{theoRKHS}, we show the case of the linear-quadratic regulator (LQR) below. 
	%%%%
	\vspace{-1em}
	\paragraph{Special case: linear-quadratic regulator}
	Consider a {\em linear} feedback $\phi_{\rm LQR}$, i.e., $\phi_{\rm LQR}(x)=-F_{\rm LQR}x,\;F_{\rm LQR}\in\Real^{n_u\x n_x}$, and a linear system $\dot{x}:=\frac{dx}{dt}=A_{\rm LQR}x+B_{\rm LQR}u$, where $A_{\rm LQR}\in\Real^{n_x\x n_x}$ and $B_{\rm LQR}\in\Real^{n_x \x n_u}$ are matrices.
	In this case, we know that the value function $V^{\phi_{\rm LQR}}$ becomes quadratic with respect to the state variable (cf. \cite{zhou1996robust}). 
	Therefore, we employ an RKHS with a quadratic kernel for $\H_V$, i.e., $\kappa^V(x,y)=(x^{\T}y)^2$.
	If we assume that the input space $\X$ is so {\em large} that the set ${\rm span}\{A_{\rm sym}|A_{\rm sym}=xx^{\T},~\exists x\in\X\}$ accommodates any
	real symmetric matrix, we obtain $\H_V=\{\X\ni x\mapsto x^{\T}A_{\rm sym}x|A_{\rm sym}~{\rm is}~{\rm symmetric}\}$.
	Moreover, it follows from the product rule of the directional derivative \cite{bonet1997nonlinear} that $K(x,y)=-x^{\T}\overline{A}_{\rm LQR}yx^{\T}y-x^{\T}yx^{\T}\overline{A}_{\rm LQR}y=x^{\T}(-\overline{A}_{\rm LQR}yy^{\T}-yy^{\T}\overline{A}^{\T}_{\rm LQR})x$, where $\overline{A}_{\rm LQR}:=A_{\rm LQR}-B_{\rm LQR}F_{\rm LQR}$.
	Note $A_{\rm value}(y):=-\overline{A}_{\rm LQR}yy^{\T}-yy^{\T}\overline{A}^{\T}_{\rm LQR}$ is symmetric, implying $K(\cdot,y)\in\H_{V}$, and
	we obtain
	$\kappa(x,y)=-x^{\T}(\overline{A}^{\T}_{\rm LQR}A_{\rm value}(y)+A_{\rm value}(y)\overline{A}_{\rm LQR})x$.  Because $A_{\rm cost}(y):=-\overline{A}^{\T}_{\rm LQR}A_{\rm value}(y)-A_{\rm value}(y)\overline{A}_{\rm LQR}$ is symmetric, it follows that $\kappa(\cdot,y)\in\H_{V}$.
	If $\overline{A}_{\rm LQR}$ is stable (Hurwitz), from Theorem \ref{theoRKHS}, the one-to-one correspondence between $\H_{V}$ and $\H_{R}$ thus implies that $\H_{V}=\H_{R}$. 
	Therefore, we can fully approximate 
	the immediate cost function $R^{\phi_{\rm LQR}}$ in $\H_R$ if $R^{\phi_{\rm LQR}}$ is quadratic with respect to the state variable.
	Suppose that the immediate cost function is given by $R^{\phi_{\rm LQR}}(x)=\sum_{i=1}^r c_i\kappa(x,x_i)=x^{\T}\overline{A}_{\rm cost}x$.
	Then, the estimated value function will be given by $V^{\phi_{\rm LQR}}(x)=U^{-1}(R^{\phi_{\rm LQR}})(x)=\sum_{i=1}^r c_iK(x,x_i)=-x^{\T}\overline{A}_{\rm value}x$,
	where
	$
	\overline{A}^{\T}_{\rm LQR}\overline{A}_{\rm value}+\overline{A}_{\rm value}\overline{A}_{\rm LQR}+\overline{A}_{\rm cost}=0,
	$
	which is indeed the well-known Lyapunov equation \cite{zhou1996robust}.
	Unlike Gaussian-kernel cases, we only require a {\em finite} number of parameters to fully approximate the immediate cost function, and hence is analytically obtainable.
	\vspace{-1em}
	\paragraph {Barrier-certified policy updates under CT formulation}
	Next, we show that the CT formulation makes barrier-certified policy updates computationally more efficient under certain conditions.
	Assume that the system dynamics is affine in the control, i.e.,
	$h(x,u)=f(x)+g(x)u$, and $\eta=0$,
	where $f:\Real^{n_x}\rightarrow\Real^{n_x}$ and $g:\Real^{n_x}\rightarrow\Real^{n_x\x n_u}$,
	and that the immediate cost $R(x,u)$ is given by
	$Q(x)+\frac{1}{2}u^{\T}Mu$,
	where $Q:\Real^{n_x}\rightarrow\Real$, and $M\in\Real^{n_u\x n_u}$ is a positive definite matrix.
	Then, any Lipschitz continuous policy $\phi:\X\rightarrow\U$ satisfying
	$
	\phi(x)\in\mathcal{S}(x):=\left\{u\in\U\mid\frac{\partial b(x)}{\partial x}f(x)+\frac{\partial b(x)}{\partial x}g(x)u+\alpha(b(x))\geq 0\right\} %\label{safe}
	$
	renders the set $\mathcal{C}$ forward invariant \cite{xu2015robustness}, i.e., the state trajectory
	remains inside the set $\mathcal{C}$, where $\alpha:\Real\rightarrow\Real$ is strictly increasing and $\alpha(0)=0$.
	Taking this constraint into account, the barrier-certified greedy policy update is given by
	\begin{align}
	\phi^+(x)=\displaystyle\argmin_{u\in\mathcal{S}(x)}\left[\frac{1}{2}u^{\T}Mu+{\frac{\partial V^{\phi}(x)}{\partial x}}g(x)u\right], \label{poliplus}
	\end{align}
	which is, by virtue of the CT formulation, a quadratic programming (QP) problem at $x$ when $\U\subset\Real^{n_u}$ defines affine constraints (see Appendix C in the supplementary document).
	The space of allowable policies is thus given by $\Gamma:=\{\phi\in\Phi\mid\phi(x)\in\mathcal{S}(x),\;\forall x\in\X\}$.
	When $\eta\neq0$ and the dynamics is learned by GPs as in \cite{pan2014probabilistic}, the work in \cite{wang2017safe} provides
	a barrier certificate for uncertain dynamics.
	Note, one can also employ a function approximator or add noises to the greedily updated policy to avoid unstable performance and promote exploration (see e.g., \cite{doya2000reinforcement}).
	To see if the updated policy $\phi^+$ remains in the space of Lipschitz continuous policies $\Phi$, i.e., $\Gamma\subset\Phi$,
	we present the following proposition.
	\begin{proposition}
		\label{LipQP}
		Assume the conditions in Theorem \ref{theoRKHS}.
		Assume also that $\U\subset\Real^{n_u}$ defines affine constraints, and that $f$, $g$, $\alpha$, and the derivative of $b$ are Lipschitz continuous over $\X$.  
		Then, the policy $\phi^+$ defined in \refeq{poliplus} is Lipschitz continuous over $\X$
		if the \emph{width of a feasible set}\footnote{\emph{Width} indicates how much control margin is left for the strictest constraint, as defined in {\cite[Equation~21]{morris2013sufficient}}.} is strictly larger than zero over $\X$.
	\end{proposition}
	\vspace{-1em}
	\begin{proof}
		See Appendix D in the supplementary document.
	\end{proof}	
	\vspace{-1em}
	Note, if $\U\subset\Real^{n_x}$ defines the bounds of each entry of control inputs, it defines affine constraints.
	Lastly, the width of a feasible set is strictly larger than zero if $\U$ is sufficiently large and $\frac{\partial b(x)}{\partial x}g(x)\neq0$.
	%%%
	
	%%%
	We will further clarify the relations of the proposed theoretical framework to existing works below.
	%%%%%%
	\section{Relations to existing works}
	First, our proposed framework takes advantage of the capability of learning complicated functions and nonparametric flexibility of RKHSs,
	and reproduces some of the existing {\em model-based} DT-VF approximation techniques (see Appendix E in the supplementary document).
	Note that some of the existing DT-VF approximations in RKHSs, such as GPTD \cite{engel2005reinforcement}, also work for model-free cases
	(see \cite{ohnishi2018safety} for model-free adaptive DT action-value function approximation, for example). 
	Second, since the RKHS $\H_{R}$ for learning is explicitly defined in our framework, any kernel-based method and its convergence/tracking analyses are
	directly applicable.
	While, for example, the work in \cite{koppel2017parsimonious}, which aims at attaining a sparse representation of the unknown function
	in an online fashion in RKHSs, was extended to the policy evaluation \cite{koppel2017policy} by addressing the double-sampling problem,
	our framework does not suffer from the double-sampling problem, and hence any kernel-based online learning (e.g., \cite{koppel2017parsimonious, yukawa_tsp12, tayu_tsp15}) can be straightforwardly applied.
	Third, when the time interval is small, DT formulations become susceptible to errors, while CT formulations are immune to the choice of the time interval.
	Note, on the other hand, a larger time interval poorly represents the system dynamics evolving in continuous time.
	Lastly, barrier certificates are efficiently incorporated in our CT framework through QPs under certain conditions, and state constraints are implicitly taken into account.
	Stochastic optimal control such as the work in \cite{theodorou2010stochastic,theodorou2010reinforcement} requires sample trajectories over predefined finite time horizons and the value
	is computed along the trajectories while the VF is estimated in an RKHS even without having to follow the trajectory in our framework.
	%%%%%%
	\section{Applications and practical implementation}
	\label{sec:applications}
	We apply the theory presented in the previous section to the Gaussian kernel case and
	introduce CTGP as an example, and
	present a practical implementation.
	Assume that $A(x,u)\in\Real^{n_x\x n_x}$ is diagonal, for simplicity.
	The Gaussian kernel is given by
	$\kappa^V(x,y):=
	\dfrac{1}{(2\pi\sigma^2)^{L/2}}
	\exp\left(-\dfrac{\norm{x - y}_{\real^{n_x}}^2}
	{2\sigma^2}\right)$, $\;x,y\in\X$, $\sigma>0$.
	Given Gaussian kernel $\kappa^V(x,y)$, the reproducing kernel $\kappa(x,y)$ defined in \refeq{kernelCT} is derived as
	$\kappa(x,y)=a(x,y)\kappa^V(x,y)$,
	where $a:\X\x\X\rightarrow\Real$ is a real-valued function (see Appendix F in the supplementary document for the explicit form of $a(x,y)$).
	\vspace{-1em}
	\paragraph {CTGP}
	One of the celebrated properties of GPs is their Bayesian formulation, which
	enables us to deal with uncertainty through credible intervals.
	Suppose that the observation $d$ at time instant $n\in\integer_{\geq0}$ contains some noise $\epsilon\in\Real$, i.e.,
	$
	d(x)=R^{\phi}(x)+\epsilon,\;\;\epsilon\sim\mathcal{N}(0,\mu_o^2),\;\mu_o\geq0.
	$
	Given data $d_N:=[d(x_0),d(x_1),\ldots,d(x_N)]^{\T}$ for some $N\in\integer_{\geq0}$, we can employ GP regression to obtain
	the mean $m(x_{*})$ and the variance $\mu^2(x_{*})$ of $\hat{R}^{\phi}(x_{*})$ at a point $x_{*}\in\X$ as
	\begin{align}
	m(x_{*})=k_{*}^{\T}(G+\mu_o^2I)^{-1}d_N,\;\;\;\;\;\mu^2(x_{*})=\kappa(x_{*},x_{*})-k_{*}^{\T}(G+\mu_o^2I)^{-1}k_{*},
	\end{align}	
	where $I$ is the identity matrix, $k_{*}:=[\kappa(x_{*},x_0),\kappa(x_{*},x_1),\ldots,\kappa(x_{*},x_N)]^{\T}$, and the $(i,j)$ entry of $G\in\Real^{(N+1)\x (N+1)}$ is $\kappa(x_{i-1},x_{j-1})$.
	Then, by the existence of the inverse operator $U^{-1}$, the mean $m^V(x_{*})$ and the variance ${\mu^V}^2(x_{*})$ of $\hat{V}^{\phi}(x_{*})$ at a point $x_{*}\in\X$ can be given by
	\begin{align}
	\hspace{-1em}m^V(x_{*})={K^V_{*}}^{\T}(G+\mu_o^2I)^{-1}d_N,\;{\mu^V}^2(x_{*})=\kappa^V(x_{*},x_{*}) -{K_{*}^V}^{\T}(G+\mu_o^2I)^{-1}K_{*}^V, \label{meancov}
	\end{align}	
	where $K^V_{*}:=[K(x_{*},x_0),K(x_{*},x_1),\ldots,K(x_{*},x_N)]^{\T}$ (see Appendix G in the supplementary document for more details).
	%%%%%%	
	\section{Numerical Experiment}
	In this section, we first show the validity of the proposed CT framework and its advantage over DT counterparts when the time interval is small,
	and then compare CTGP and GPTD for RL on a simulated inverted pendulum.
	In both of the experiments, the coherence-based sparsification \cite{richard09} in the RKHS $\H_{R}$ is employed to curb the growth of the dictionary size.
	\vspace{-1em}
	\paragraph {Policy evaluations: comparison of CT and DT approaches}
	We show that our CT approaches are advantageous over DT counterparts in terms of susceptibility to errors, by using
	MountainCarContinuous-v0 in OpenAI Gym \cite{openaigym} as the environment.
	The state is given by $x(t):=[{\rm x}(t),{\rm v}(t)]^{\T}\in\Real^2$, where ${\rm x}(t)$ and ${\rm v}(t)$ are the position and the velocity of the car, and the dynamics is
	given by
	$
	dx =
	\left[
	\begin{array}{c}
	{\rm v}(t) \\
	-0.0025\cos{(3{\rm x}(t))}
	\end{array}
	\right]dt+\left[
	\begin{array}{c}
	0\\
	0.0015
	\end{array}
	\right]u(t)dt,
	$
	where $u(t)\in[-1.0,1.0]$.
	The position and the velocity are clipped to $[-0.07,0.07]$ and $[-1.2,0.6]$, respectively, and the goal is to reach the position ${\rm x}=0.45$.
	In the simulation, the control cycle (i.e., the frequency of applying control inputs and observing the states and costs) is set to $1.0$ second.
	The observed immediate cost is given by $R(x(t),u(t))+\epsilon=1+0.001u^2(t)+\epsilon$ for ${\rm x}(t)< 0.45$ and $R(x(t),u(t))+\epsilon=0.001u^2(t)+\epsilon$ for ${\rm x}(t)\geq0.45$, where $\epsilon\sim\mathcal{N}(0,0.1^2)$.  Note the immediate cost for the DT cases is given by $(R(x(t),u(t))+\epsilon)\Delta t$, where $\Delta t$ is the time interval.
	For policy evaluations, we use the policy obtained by RL based on the cross-entropy methods\footnote{We used the code in https://github.com/udacity/deep-reinforcement-learning/blob/master/cross-entropy/CEM.ipynb offered by Udacity.  The code is based on PyTorch \cite{paszke2017automatic}.}, and
	the four methods, CTGP, KF-based CT-VF approximation (CTKF), GPTD, and KF-based DT-VF approximation (DTKF),
	are used to learn value functions associated with the policy by executing the current policy for five episodes, each of which terminates
	whenever $t=300$ or ${\rm x}(t)\geq0.45$.  GP-based techniques are expected to handle the random component $\epsilon$ added to the immediate cost.
	The new policies are then obtained by the barrier-certified policy updates under CT formulations, and
	these policies are evaluated for five times.
	Here, the barrier function is given by $b(x)=0.05+{\rm v}$, which prevents the velocity from
	becoming lower than $-0.05$.
	\reffig{comparisonPE} compares the value functions\footnote{We used "jet colormaps" in Python Matplotlib for illustrating the value functions.} learned by each method for the time intervals $\Delta t=20.0$ and $\Delta t=1.0$.
	We observe that the DT approaches are sensitive to the choice of $\Delta t$. 
	\reftab{compcosts} compares the cumulative costs averaged over five episodes for each method and for different time intervals and the numbers of times we observed the velocity being lower than $-0.05$ when the barrier certificate is employed and unemployed. (Numbers associated with the algorithm names indicate the lengths of the time intervals.)
	Note that the cumulative costs are calculated by summing up the immediate costs multiplied by the duration of each control cycle, i.e., we discretized the immediate cost based on the control cycle.
	It is observed that the CT approach is immune to the choice of $\Delta t$ while the performance of the DT approach degrades
	when the time interval becomes small, and that the barrier-certified policy updates work efficiently.
	\begin{figure}[t]
		\begin{minipage}{0.33\hsize}
			\begin{center}
				\includegraphics[clip,width=0.85\textwidth]{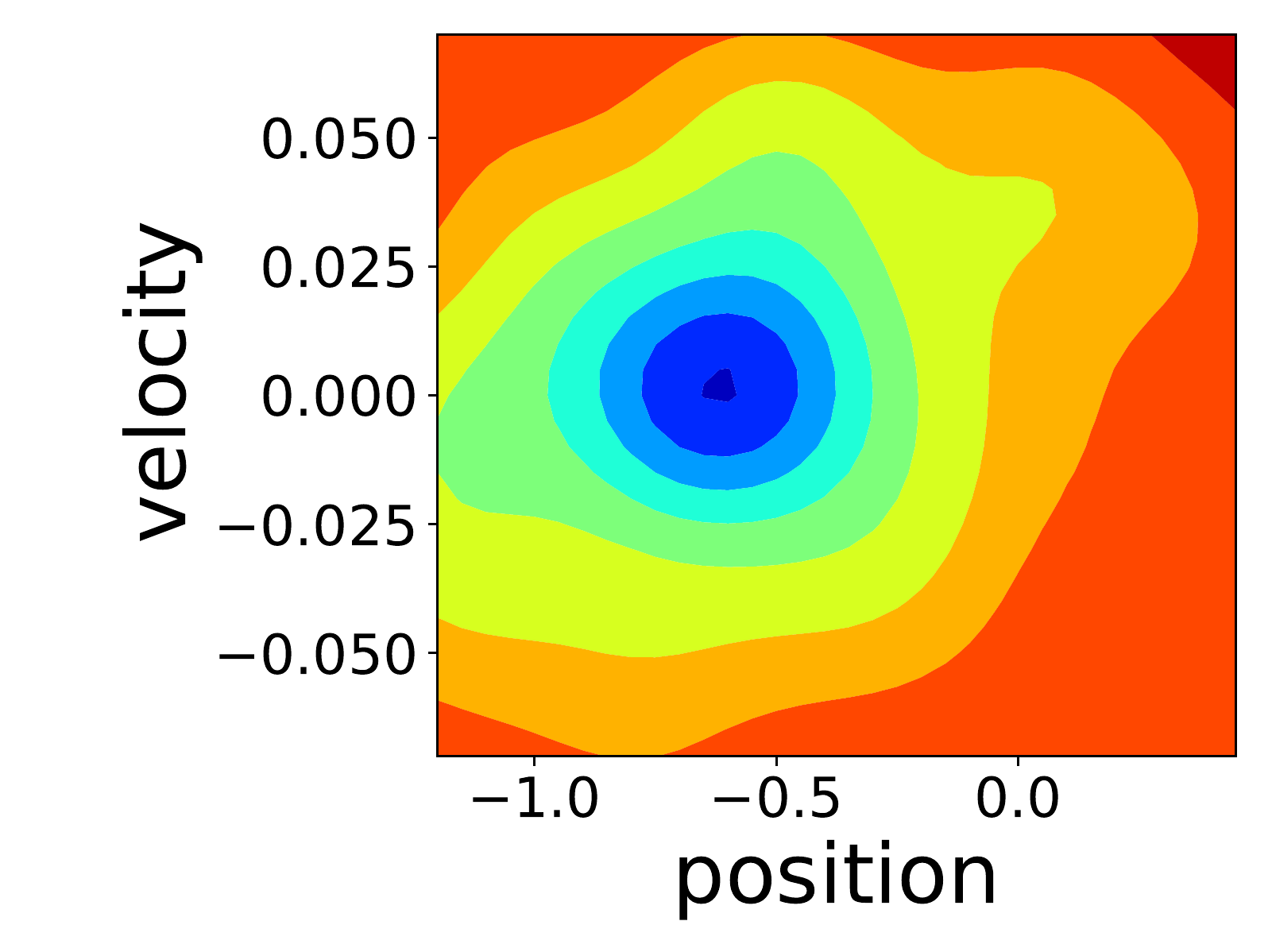}
				\label{fig:GPTD_20}
				\centerline{(a) GPTD for $\Delta t=20.0$}
			\end{center}	
		\end{minipage}
		\begin{minipage}{0.33\hsize}
			\begin{center}
				\includegraphics[clip,width=0.85\textwidth]{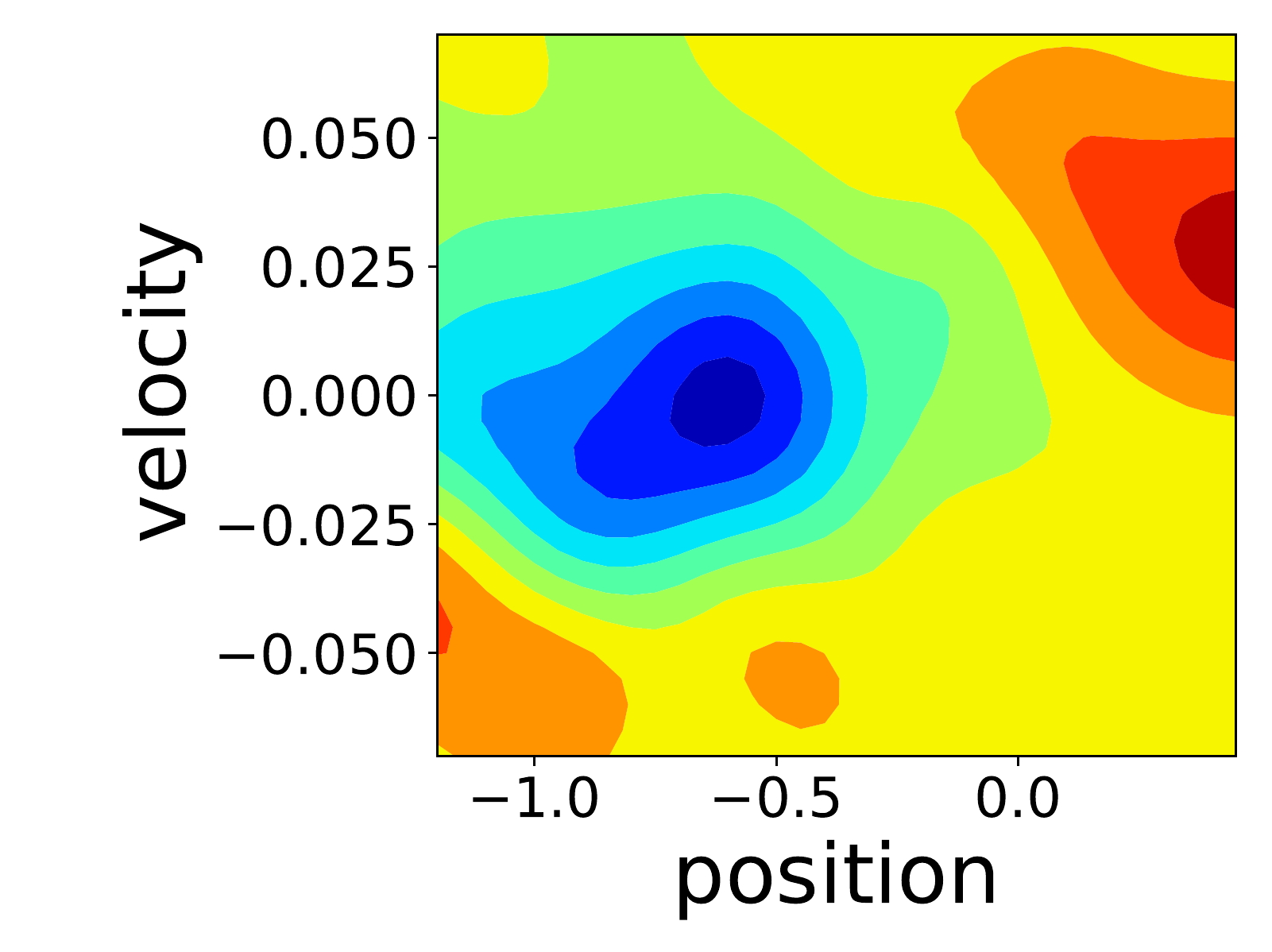}
				\label{fig:GPTD_1}
				\centerline{(b) GPTD for $\Delta t=1.0$}
			\end{center}
		\end{minipage}
		\begin{minipage}{0.33\hsize}
			\begin{center}
				\includegraphics[clip,width=0.85\textwidth]{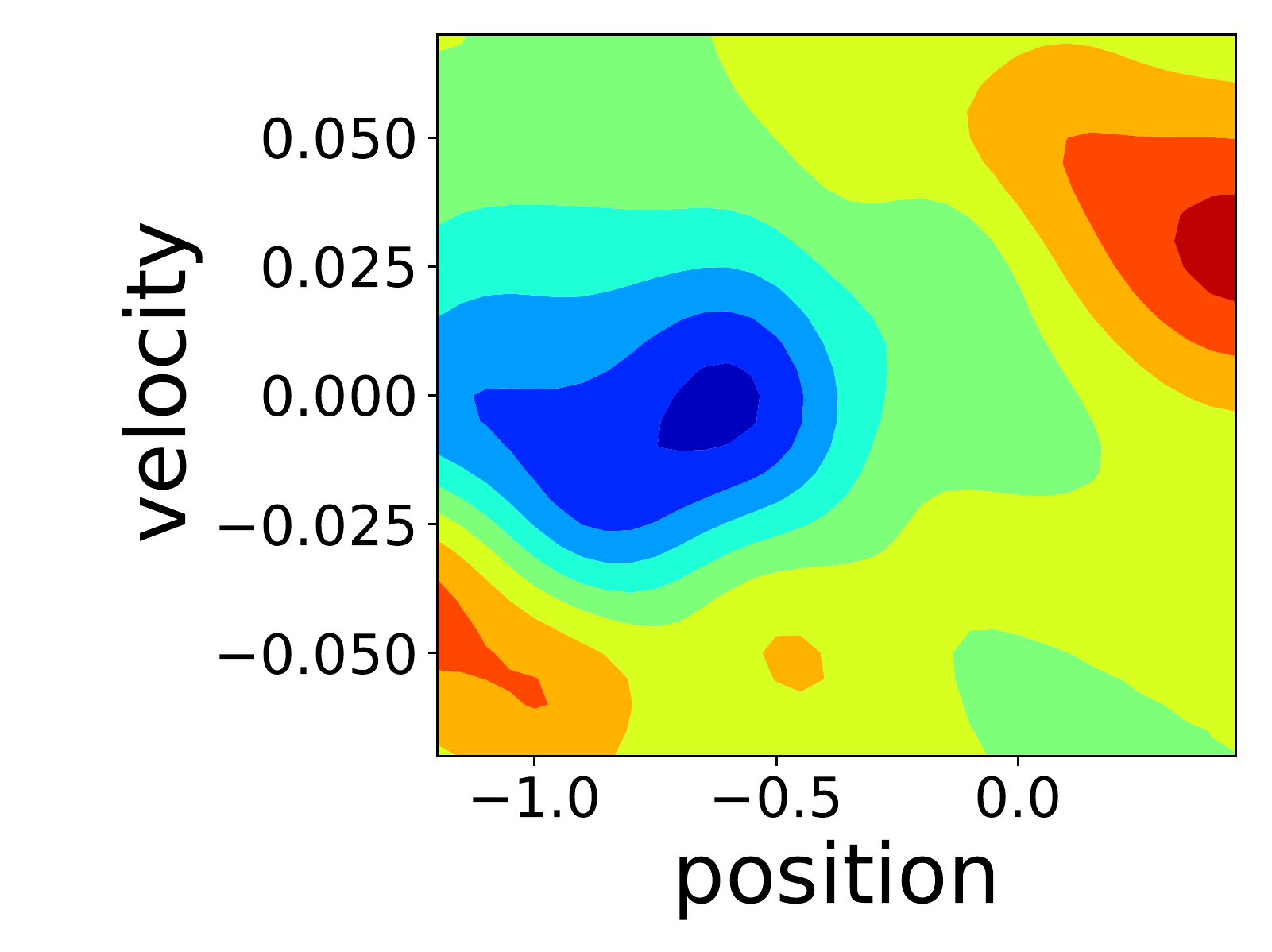}
				\label{fig:CTGP}
				\centerline{(c) CTGP}
			\end{center}
		\end{minipage}
		\begin{minipage}{0.33\hsize}
			\begin{center}
				\includegraphics[clip,width=0.85\textwidth]{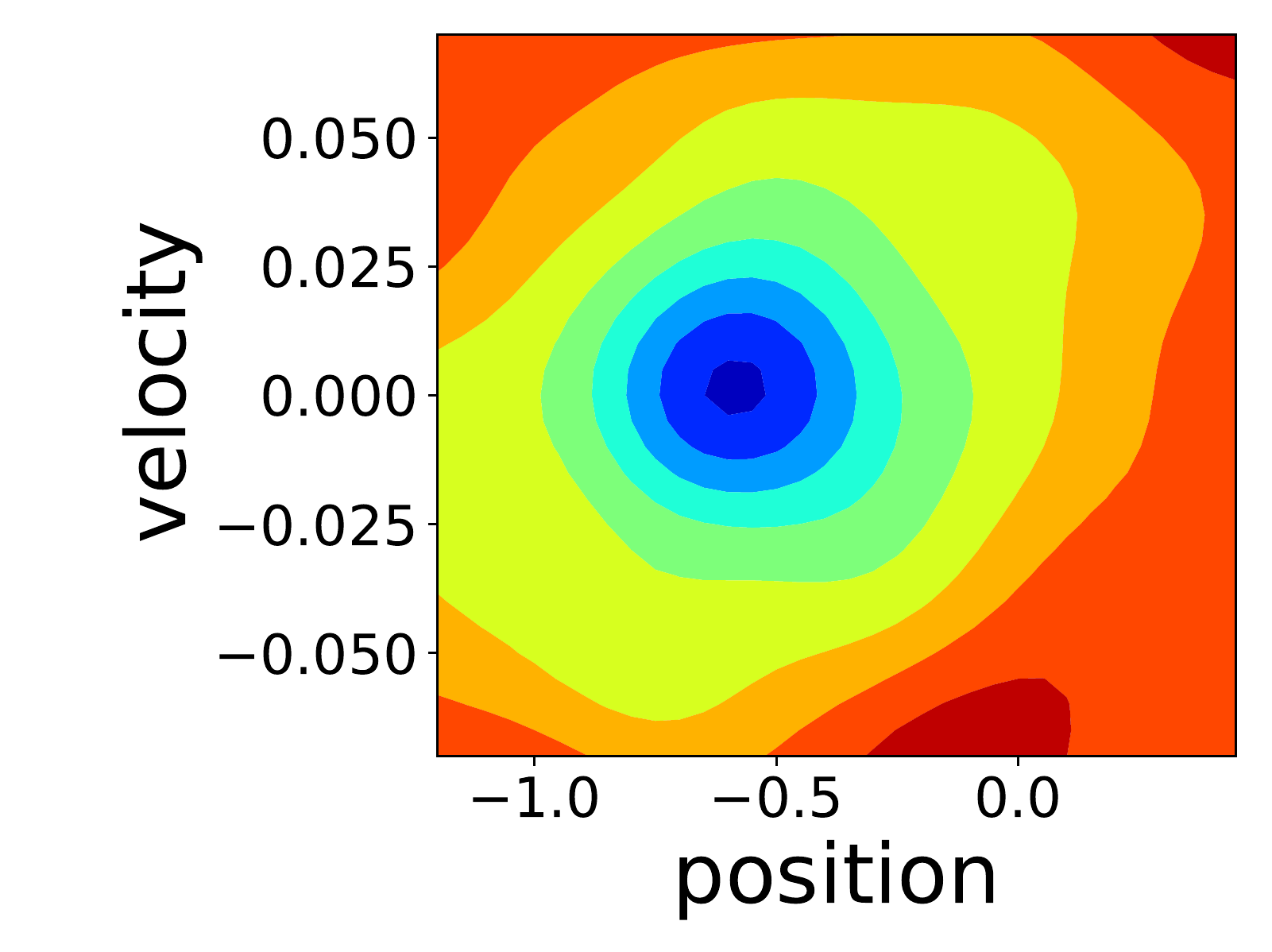}
				\label{fig:DTKF_20}
				\centerline{(d) DTKF for $\Delta t=20.0$}
			\end{center}
		\end{minipage}
		\begin{minipage}{0.33\hsize}
			\begin{center}
				
				\includegraphics[clip,width=0.85\textwidth]{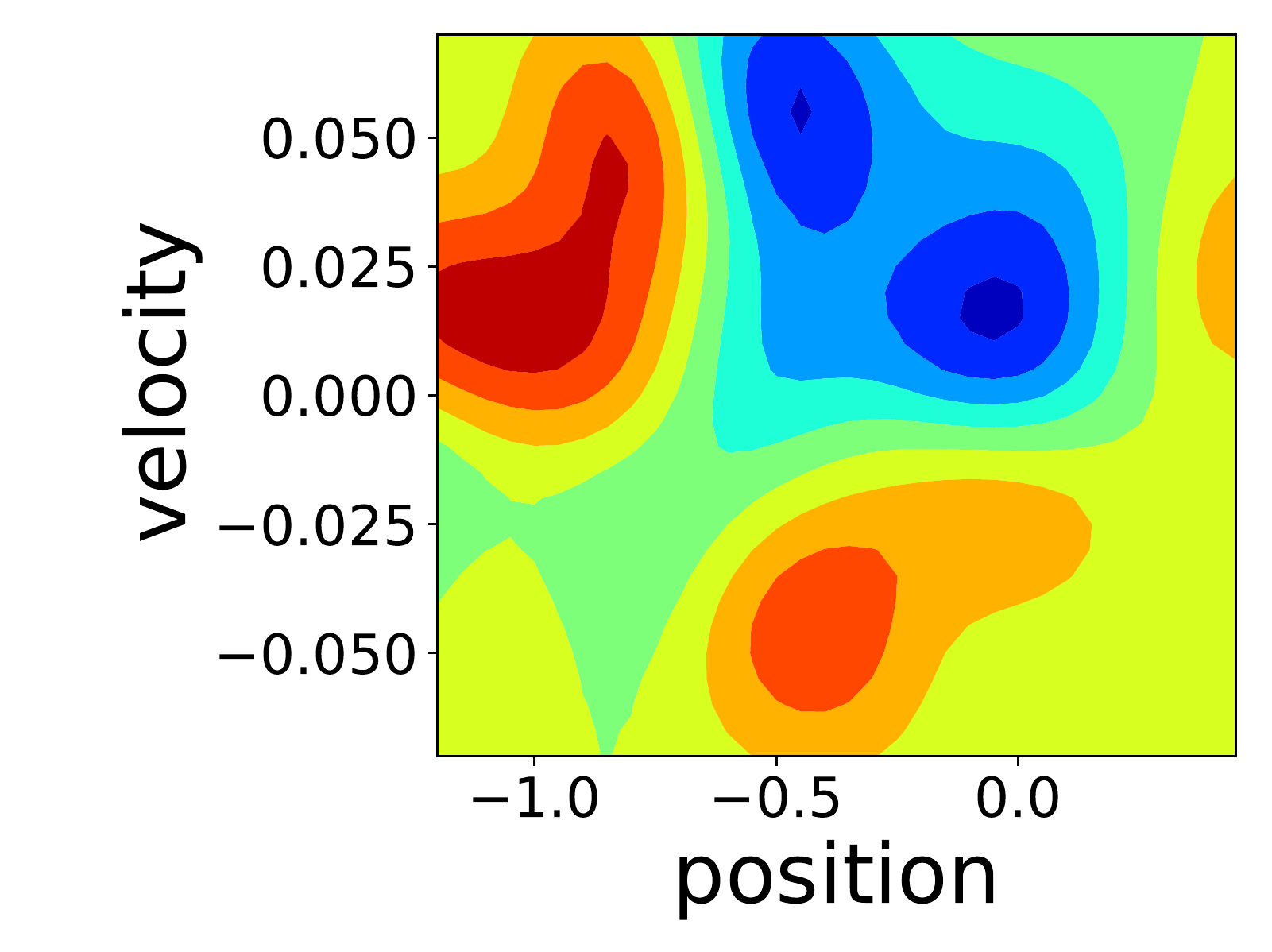}
				\label{fig:DTKF_1}
				\centerline{(e) DTKF for $\Delta t=1.0$}
			\end{center}
		\end{minipage}
		\begin{minipage}{0.33\hsize}
			\begin{center}
				\hspace{3em}
				\hspace{1em}
				\includegraphics[clip,width=0.85\textwidth]{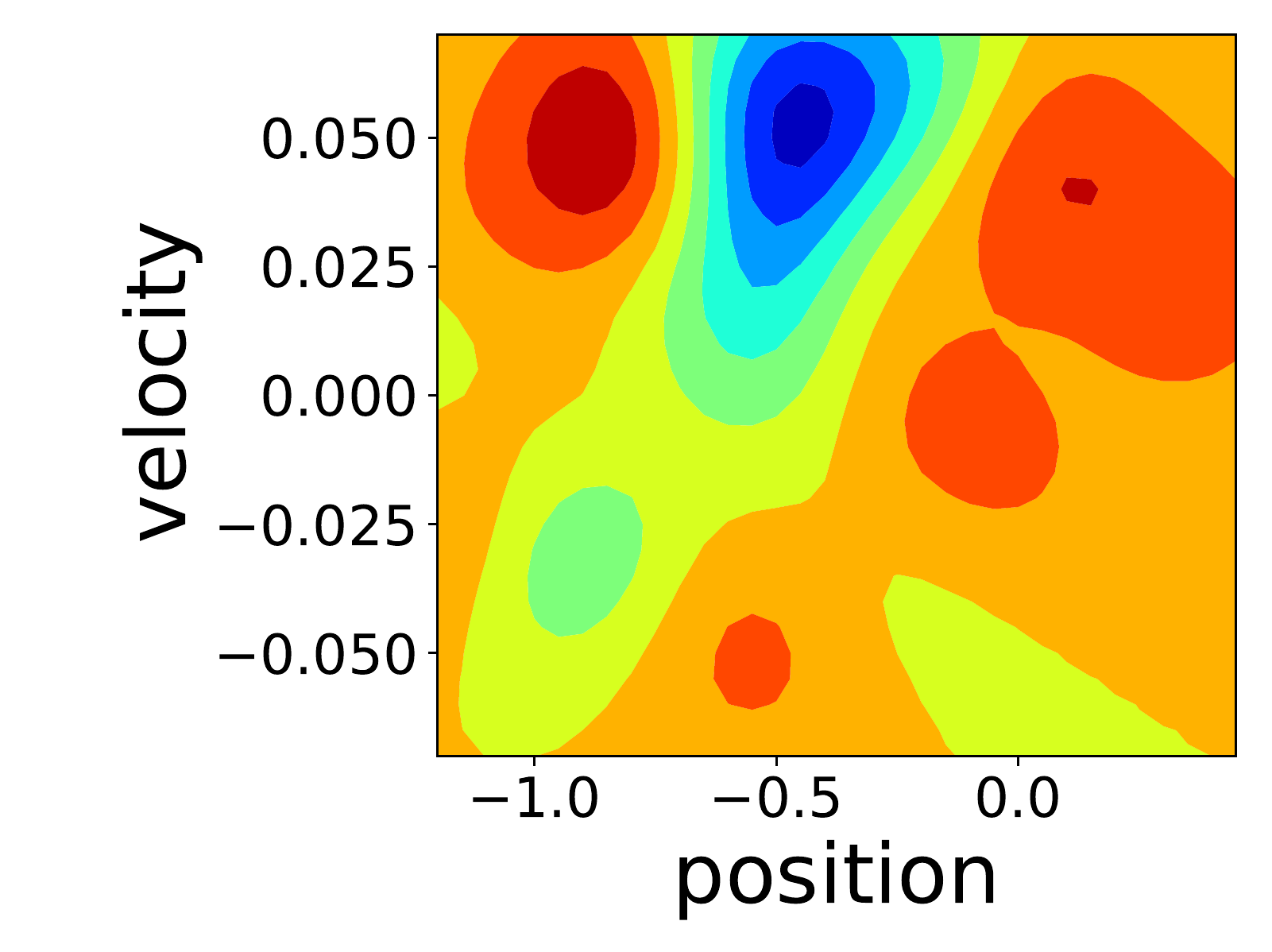}
				\label{fig:CTKF}
				\centerline{(f) CTKF}
			\end{center}
		\end{minipage}
		\caption{Illustrations of the value functions obtained by CTGP, CTKF, GPTD, and DTKF for time intervals $\Delta t=20.0$ and $\Delta t=1.0$.
			The policy is obtained by RL based on the cross-entropy method.  CT approaches are not affected by the choice of $\Delta t$.}
		\label{comparisonPE}
	\end{figure}
	%%%%%%
	%\begin{minipage}{0.65\hsize}
	\begin{table}[t]
		\caption{Comparisons of the cumulative costs and numbers of times the observed velocities became lower than $-0.05$ with and without barrier certificates}
		\label{compcosts}
		\centering
		\begin{tabular}{lllllll}
			\toprule
			& CTKF & GPTD\_1 & DTKF\_1 & CTGP & GPTD\_20 & DTKF\_20 \\
			\midrule
			Cumulative cost & $114.2$ & $299.1$ & $299.1$ & $82.2$ & $89.2$& $90.4$ \\
			With barrier & $0$ (times) & $0$ (times) & $0$ (times) & $0$ (times) & $0$ (times)& $0$ (times) \\
			Without barrier & $0$ (times) & $0$ (times) & $0$ (times) & $10$ (times)& $20$ (times)& $20$ (times) \\
			\bottomrule
		\end{tabular}
	\end{table}
	%\end{minipage}
	%\begin{minipage}{0.33\hsize}
	%\end{minipage}
	\vspace{-1em}
	\paragraph {Reinforcement learning: inverted pendulum}
	\begin{wrapfigure}{R}{0.35\textwidth}
		%\begin{center}
		\vspace{-5em}
		\includegraphics[clip,width=0.35\textwidth]{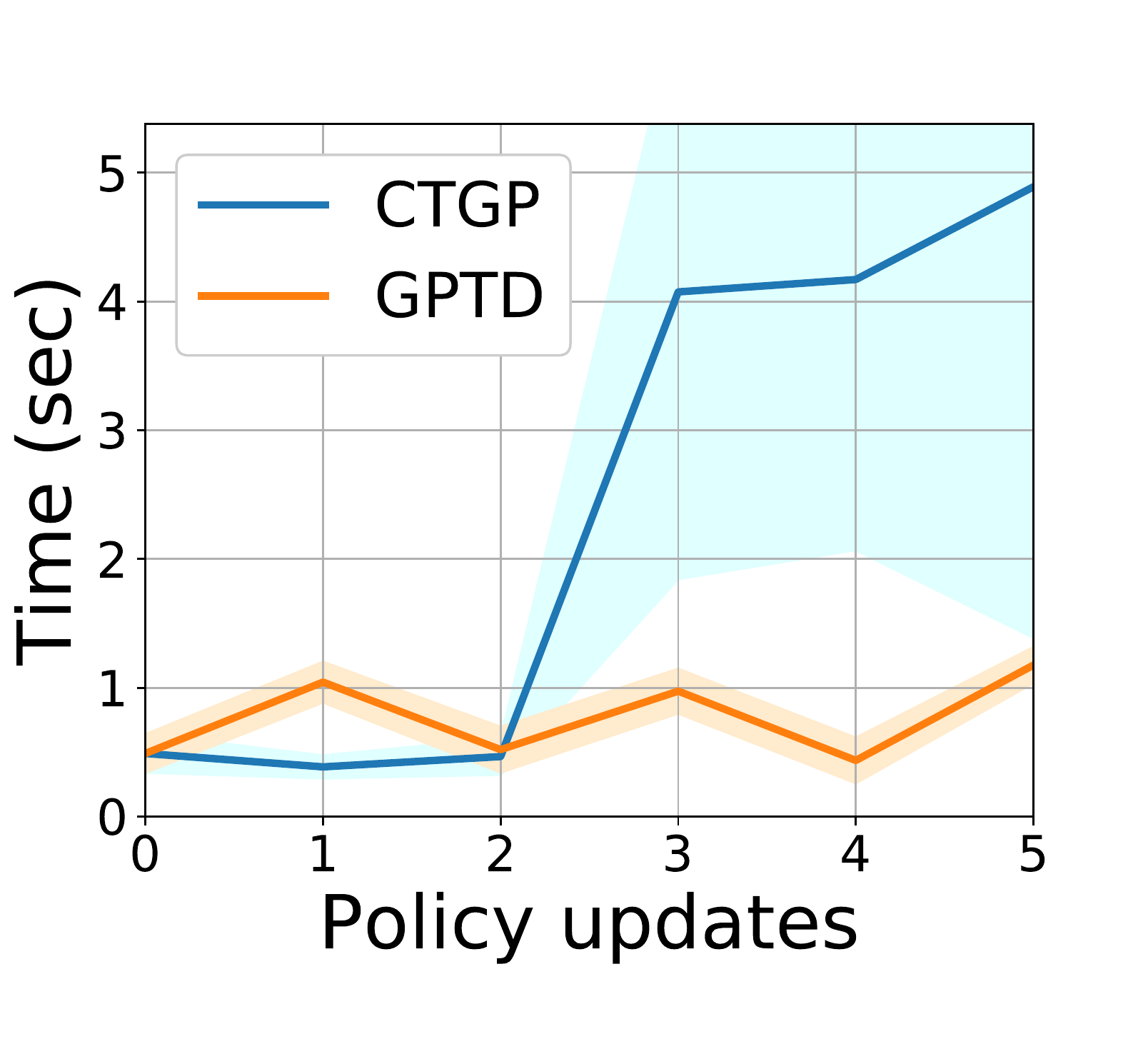}
		\label{fig:Pendulum}
		%\end{center}	
		\caption{Comparison of time up to which the pendulum stays up between CTGP and GPTD for the inverted pendulum ($\pm$ std. deviation).}
		\label{comparison}
	\end{wrapfigure}
	We show the advantage of CTGP over GPTD when the time interval for the estimation is small.
	Let the state $x(t):=[\theta(t),\omega(t)]^{\T}\in\Real^2$ consists of the angle $\theta(t)$ and the angular velocity $\omega(t)$ of an inverted pendulum, and we consider the dynamics: 
	$
	dx = 
	\left[
	\begin{array}{c}
	\omega(t) \\
	\frac{g}{\ell}{\rm sin}(\theta(t))-\frac{\rho}{m\ell^2}\omega(t)
	\end{array}
	\right]dt+\left[
	\begin{array}{c}
	0\\
	\frac{1}{m\ell^2}
	\end{array}
	\right]u(t)dt+
	0.01Idw,
	$
	where $g=9.8,\;m=1,\;\ell=1,\;\rho=0.01$.
	The Brownian motion may come from outer disturbances and/or modeling error.
	In the simulation, the time interval $\Delta t$ is set to $0.01$ seconds, and the simulated dynamics evolves by
	$\Delta x=h(x(t),u(t))\Delta t+\sqrt{\Delta t}\eta(x(t),u(t))\epsilon_w$, where $\epsilon_w\sim\mathcal{N}(0,I)$.
	%The discount factors for CTGP and GPTD are set to $0.01$ and $e^{-0.01\Delta t}=0.9999$, respectively.
	In this experiment, the task is to stabilize the inverted pendulum at $\theta=0$.
	The observed immediate cost is given by $R(x(t),u(t))+\epsilon=1/(1+e^{-10(\theta(t)-\pi/16)})+100/(1+e^{-10(\theta(t)-\pi/6)})+0.05u^2(t)+\epsilon$, where $\epsilon\sim\mathcal{N}(0,0.1^2)$.
	A trajectory associated with the current policy is generated to learn the VF. 
	The trajectory terminates when $|\theta(t)|>\pi/4$ and restarts from a random initial angle.
	After $10$ seconds, the policy is updated.
	To evaluate the current policy, average time over five episodes in which the pendulum stays up ($|\theta(t)|\leq\pi/4$) when initialized as
	$\theta(0)\in[-\pi/6,\pi/6]$ is used.  \reffig{comparison} compares this average time of CTGP and GPTD up to five updates
	with standard deviations until when stable policy improvement becomes difficult without some heuristic techniques such as adding noises to policies.
	Note that the same seed for the random number generator is used for the initializations of both of the two approaches.
	It is observed that GPTD fails to improve policies.  The large credible interval of CTGP is due to the random initialization of the state.
	
	\section{Conclusion and future work}
	We presented a novel theoretical framework that renders the CT-VF approximation problem 
	simultaneously solvable in an RKHS by conducting 
	kernel-based supervised learning for the immediate cost function in
	the properly defined RKHS.
	Our CT framework is compatible with rich theories of control, including control barrier certificates for safety-critical applications.
	The validity of the proposed framework and its advantage over DT counterparts when the time interval is small were verified experimentally
	on the classical Mountain Car problem and a simulated inverted pendulum.
	%%%
	
	%%%
	There are several possible directions to explore as future works; First, we can employ the state-of-the-art kernel methods
	within our theoretical framework or use other variants of RL, such as actor-critic methods, to improve practical performances.
	Second, we can consider uncertainties in value function approximation by virtue of the RKHS-based formulation, which
	might be used for safety verifications.
	Lastly, it is worth further explorations of advantages of CT formulations for physical tasks.
	\clearpage
	\subsubsection*{Acknowledgments}
	This work was partially conducted when M. Ohnishi was at the GRITS Lab, Georgia Institute of Technology;
	M. Ohnishi thanks the members of the GRITS Lab, including Dr. Li Wang, and Prof. Magnus Egerstedt for discussions regarding
	barrier functions.
	M. Yukawa was supported in part by KAKENHI 18H01446 and 15H02757,
	M. Johansson was supported in part by the Swedish Research Council and by the Knut and Allice Wallenberg Foundation, and
	M. Sugiyama was supported in part by KAKENHI 17H00757.
	Lastly, the authors thank all of the anonymous reviewers for their very insightful comments.
	
		\begin{appendices}
		\numberwithin{equation}{section}
		\numberwithin{theorem}{section}
		\numberwithin{lemma}{section}
		\numberwithin{proposition}{section}
		\numberwithin{remark}{section}
		\numberwithin{corollary}{section}
		\numberwithin{table}{section}
		\numberwithin{figure}{section}
		\numberwithin{fact}{section}
		\renewcommand{\theequation}{\thesection.\arabic{equation}}
		%%%%%%
		%%%%%%
		\section{Tools to prove Theorem 1}
		Some known properties of RKHSs and Dynkin's formula which will be used to prove Theorem 1 are given below.
		\begin{lemma}[{\cite[Theorem~2]{minh2010some}}]
			\label{hojo1}
			Let $\X\subset\Real^{n_x}$ be any set with nonempty interior.  Then, the RKHS associated with the Gaussian kernel for an arbitrary scale parameter $\sigma>0$
			does not contain any polynomial on $\X$, including the nonzero constant function.
		\end{lemma}
		\begin{proposition}[{\cite[Theorem~1]{zhou2008derivative}}]
			\label{theoderiv}
			Let $(\H,\inpro{\cdot,\cdot}_{\H})$ be the RKHS associated with a Mercer kernel $\kappa\in C^{2s}(\X\x\X),\;s\in\integer_{\geq0}$, where $\X\subset\Real^{n_x}$ is compact with nonempty interior.
			Then, $(D^{\alpha}\kappa)_x\in\H,\;\forall x\in\X,\;\alpha\in I_s$, and
			\begin{align}
				D^{\alpha}\varphi(x)=\inpro{(D^{\alpha}\kappa)_x,\varphi}_{\H},\;\forall x\in\X,\;\varphi\in\H. \label{parrepro}
			\end{align}
		\end{proposition}
		\paragraph {Dynkin's formula}
		Under Assumption 1, we obtain
		Dynkin's formula (cf. {\cite[Theorem~7.3.3,~Theorem~7.4.1]{oksendal2003stochastic}}):
		\begin{align}
			E_{x}\left[\Psi(x(t_1))\right]-\Psi(x)=-E_{x}\left[\int_{0}^{t_1}\mathcal{G}(\Psi)(x(q))dq\right],\;\;\forall t_1\in[0,\infty), \label{dynkins}
		\end{align}
		for any $x\in\Real^{n_x}$ and for any $\Psi\in C_0^2(\Real^{n_x})$, i.e., $\Psi\in C^2(\Real^{n_x})$ and $\Psi$ has compact support, where $\mathcal{G}$ is defined in (5).
		Moreover, it holds {\cite[Chapter~III.3]{fleming2006controlled}}, for $\beta>0$, that
		\begin{align}
			e^{-\beta t_1}E_{x}\left[\Psi(x(t_1))\right]-\Psi(x)=-E_{x}\left[\int_{0}^{t_1}e^{-\beta q}\left[\beta I_{\rm op}+\mathcal{G}\right](\Psi)(x(q))dq\right],\;\;\forall t_1\in[0,\infty). \label{dynkins2}
		\end{align}
		When $t_1$ is the first exit time of a bounded set\footnote{The first exit time $t_1$ of a bounded set ${\rm int}(\X)$ is given by $t_1:=\inf\{q\mid x(q)\neq {\rm int}(\X)\}$ starting from a point $x(0)=x\in{\rm int}(\X)$.}, then the condition for $\Psi$ is weakened into $\Psi\in C^2$ over the bounded set (see the remark of  {\cite[Theorem~7.4.1]{oksendal2003stochastic}}).
		See \reffig{fig:illustdynkin} for an intuition of the expectation $E_x$ taken for the trajectories of the state starting from $x$.
		%%%%%%%%%
		
		\section{Proof of Theorem 1}
		\label{proof}
		Note first that the operator $U$ is well defined because $D^{\alpha}\varphi^V$ exists for any $\varphi^V\in\H_V$ from Proposition \ref{theoderiv}.
		We show that $U$ is bijective linear, and then show that the reproducing kernel in $\H_{R}$ is given by (9).
		\paragraph{Proof of (a)}
		Because the operator $U$ is surjective by definition of $\H_R$, we show that $U$ is injective.
		The operator $U$ is linear because the operator $D^{\alpha}$ is linear by \refeq{parrepro} in Proposition \ref{theoderiv}:
		\begin{align}
			D^{\alpha}(\nu_1\varphi_1^V+&\nu_2\varphi_2^V)(x)=\inpro{(D^{\alpha}\kappa)_x,\nu_1\varphi_1^V+\nu_2\varphi_2^V}_{\H_V}\nn\\
			=&\nu_1\inpro{(D^{\alpha}\kappa)_x,\varphi_1^V}_{\H_V}+\nu_2\inpro{(D^{\alpha}\kappa)_x,\varphi_2^V}_{\H_V}\nn\\
			=&\left[\nu_1D^{\alpha}\varphi_1^V+\nu_2D^{\alpha}\varphi_2^V\right](x),\;\forall x\in\X,\;\forall \nu_i\in\Real,\;\forall \varphi_i\in\H_V,\;i\in\{1,2\}. 
		\end{align}
		Hence, it is sufficient to show that ${\rm ker}(U)=0$ \cite{linearalgebra}.
		Suppose that $U(\varphi^V)(x)=0,\;\forall x\in\X$.
		It follows that $\left[\beta I_{\rm op}+\mathcal{G}\right](\varphi^V)(x)=0,\;\forall x\in{\rm int}(\X)$, where
		$\mathcal{G}$ is defined in (6).  %Because $\varphi^V$ is twice continuously differentiable over ${\rm int}(\X)\subset\X$, 
		Under Assumptions 1 and 2 (i.e., compactness of $\X$),
		Dynkin's fomula \refeq{dynkins} and \refeq{dynkins2} can be applied to $\varphi^V$ over the bounded set ${\rm int}(\X)$ because $\varphi^V|_{{\rm int}(\X)}\in C^2({\rm int}(\X))$.
		(i) When the discount factor $\beta>0$, we apply \refeq{dynkins2} to $\varphi^V$.
		Under Assumption 1, we can consider the time being taken to infinity.
		Because $\left[\beta I_{\rm op}+\mathcal{G}\right](\varphi^V)(x)=0,\;\forall x\in{\rm int}(\X)$, we obtain 
		\begin{align}
			\varphi^V(x)=&\displaystyle\lim_{t_1\rightarrow\infty}E_x\left[\int_0^{t_1}e^{-\beta q}\cdot0dq\right]+\displaystyle\lim_{t_1\rightarrow\infty}e^{-\beta t_1}E_x\left[\varphi^V(x(t_1))\right]\nn\\
			=&\displaystyle\lim_{t_1\rightarrow\infty}e^{-\beta t_1}E_x\left[\varphi^V(x(t_1))\right],\;\forall x\in{\rm int}(\X).
		\end{align}
		Under Assumption 2 (i.e., compactness of $\X$ and invariance of ${\rm int}(\X)$), $\displaystyle\lim_{t_1\rightarrow\infty}E_x\left[\varphi^V(x(t_1))\right]$ is bounded, from which it follows that
		$\varphi^V(x)=0$ over ${\rm int}(\X)$.
		(ii) When $\beta=0$ and $x_{t\rightarrow\infty}$ is stochastically asymptotically stable over ${\rm int}(\X)$,
		we apply \refeq{dynkins} to $\varphi^V$.
		Because $\mathcal{G}(\varphi^V)(x)=0,\;\forall x\in{\rm int}(\X)$, we obtain
		\begin{align} 
			\varphi^V(x)=\displaystyle\lim_{t_1\rightarrow\infty}E_x\left[\varphi^V(x(t_1))\right]=\varphi^V(x_{t\rightarrow\infty}),\;\forall x\in{\rm int}(\X),
		\end{align}
		which implies that $\varphi^V$ is constant over ${\rm int}(\X)$.
		From Lemma \ref{hojo1}, however, it follows that $\varphi^V(x)=0,\;\forall x\in{\rm int}(\X)$, when $\H_V$ is a Gaussian RKHS.
		Therefore, continuity of an element of $\H_V$ implies that $\varphi^V(x)=0$ over $\X$ for both cases (i) and (ii), which
		verifies ${\rm ker}(U)=0$.
		Thus, the correspondence between $\varphi^V\in\H_V$ and $\varphi=U(\varphi^V)\in\H_R$ is one-to-one, and inner product preserves by definition (8).
		\begin{figure}[t]
			\begin{center}
				\includegraphics[clip,width=0.35\textwidth]{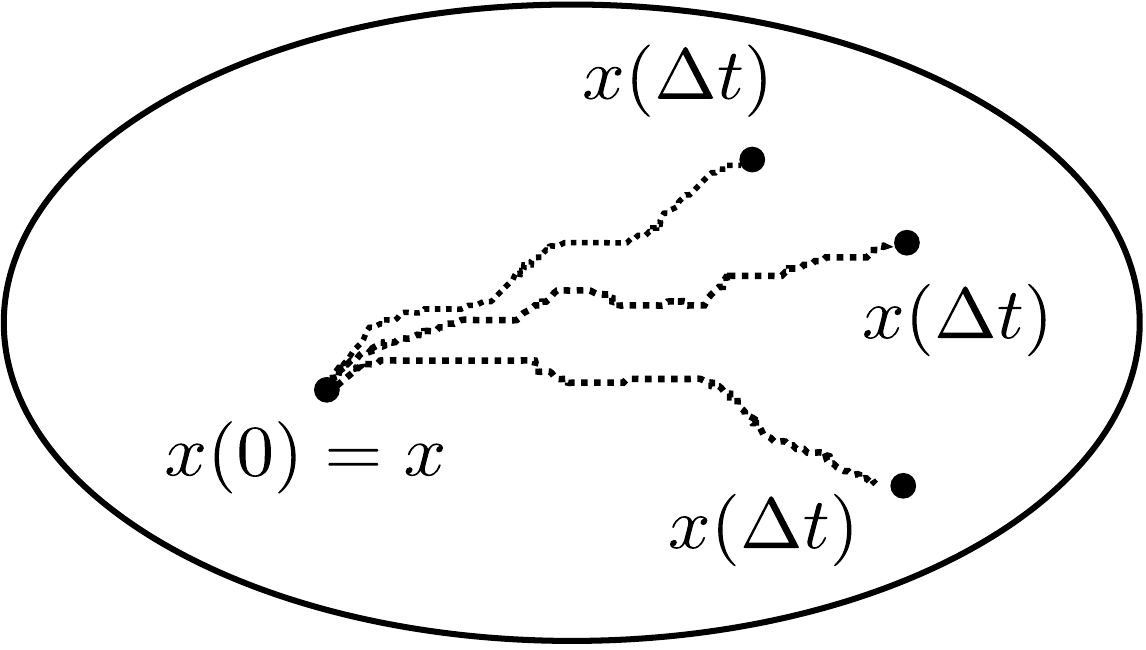}
				\caption{Given $x\in\Real^{n_x}$, the state evolves from $x(0)=x$ by following the SDE, and $x(\Delta t)$ is thus stochastic.  The expectation $E_{x}$
					is taken for all the trajectories.}
				\label{fig:illustdynkin}
			\end{center}
		\end{figure}
		\paragraph{Proof of (b)}
		We show that $\H_{R}$ is an RKHS.
		Under Assumptions 1 and 2 (i.e., compactness of $\X$), $h^{\phi}$ and $A^{\phi}$ are continuous and hence are bounded over $\X$.
		Thus, Proposition \ref{theoderiv} (i.e., $(D^{\alpha}\kappa)_x\in\H_V$) implies that $K(\cdot,x)\in\H_V$.
		Therefore, it follows that $\kappa(\cdot,x)\in\H_{R}$.
		Moreover, from \refeq{parrepro} in Proposition \ref{theoderiv}, we obtain
		\begin{align}
			U(\varphi^V)(x)=\inpro{\varphi^V,K(\cdot,x)}_{\H_V},
		\end{align}
		from which we obtain
		\begin{align}
			\inpro{\kappa(\cdot,x),\kappa(\cdot,y)}_{\H_{R}}=\inpro{K(\cdot,x),K(\cdot,y)}_{\H_V}=U(K(\cdot,y))(x)=\kappa(x,y),
		\end{align}	
		and that
		\begin{align}
			\inpro{\varphi,\kappa(\cdot,x)}_{\H_{R}}=\inpro{U^{-1}(\varphi),K(\cdot,x)}_{\H_V}=U(U^{-1}(\varphi))(x)=\varphi(x),\;\forall \varphi\in\H_{R},\forall x\in\X.
		\end{align}	
		Therefore, $\kappa(\cdot,\cdot):\X\x\X\rightarrow\Real$ is the reproducing kernel with which the RKHS $\H_{R}$ is associated.
		%%%%%%%%
		\section{Barrier-certified policy update}
		As illustrated in Figure 2, the space of the allowable policies is given by $\Gamma$ instead of $\Phi$ to
		implicitly enforce state constraints.  Therefore, the greedy policy update is conducted assuming
		that $\Gamma$ is the whole policy space.  We assume that $\eta$ and hence $A$ is independent on $u$, $h(x,u)=f(x)+g(x)u$, and the immediate cost $R(x,u)$ is given by
		$Q(x)+\frac{1}{2}u^{\T}Mu$ ($\eta$ was assumed to be $0$ in the main text for simplicity of barrier certificates).  Given the current policy $\phi$, from (4), the greedy policy update at state $x\in{\rm int}(X)$ is given by
		\begin{align}
			\phi^+(x)&=\displaystyle\argmin_{u\in\mathcal{S}(x)}\left[\frac{1}{2}{\rm tr}\left[\frac{\partial^2 V^{\phi}(x)}{\partial x^2}A(x,u)\right]+
			\frac{\partial V^{\phi}(x)}{\partial x}h(x,u)+R(x,u)\right], \\
			&=\displaystyle\argmin_{u\in\mathcal{S}(x)}\left[\frac{1}{2}u^{\T}Mu+{\frac{\partial V^{\phi}(x)}{\partial x}}g(x)u\right].
		\end{align}
		The simplicity of this optimization problem comes from the CT formulation of the VF and barrier certificates.
		%%%%%%
		\section{Lipschitz continuity of barrier-certified policies}
		Proposition 1 is based on the following theorem.
		\begin{theorem}[{\cite[Theorem~1]{morris2013sufficient}}]
			\label{QPsuf}
			Consider the QP:
			\begin{align}
				u^{*}(x)=&\displaystyle\argmin_{u\in\Real^{n_u}}\;u^{\T}H(x)u+2v(x)^{\T}u \label{LIPQP}\\
				{\rm s.t.}&\;\; A_{(\leq)}(x)u\leq a_{(\leq)}(x)\nn \\
				&\;\; A_{(=)}(x)u=a_{(=)}(x),\nn
			\end{align}	
			where $H,\;v,\;A_{(\leq)},\;A_{(=)},\;a_{(\leq)}$, and $a_{(=)}$ are continuous functions,
			and define the width of a feasible set as the unique solution to the following linear program:
			\begin{align}
				u_w(x) = &\displaystyle\max_{[u^{\T},u_w]^{\T}\in\Real^{n_u+1}} u_w\\
				{\rm s.t.}&\;\; A_{(\leq)}(x)u+[u_w,u_w,\ldots,u_w]^{\T}\leq a_{(\leq)}(x) \nn\\
				&\;\; A_{(=)}(x)u=a_{(=)}(x).\nn
			\end{align}	
			Suppose that the following conditions hold at a point $x^{*}\in\X$:
			\begin{enumerate}
				\item $u_w(x^{*})>0$
				\item $A_{(=)}(x)$ has full row rank
				\item $A_{(\leq)}(x),\;A_{(=)}(x),\;a_{(\leq)}(x),$ and $a_{(=)}(x)$ are Lipschitz continuous at $x^{*}$
				\item $H(x^{*})=H^{\T}(x^{*})$ and is positive definite
				\item $H(x)$ and $v(x)$ are Lipschitz continuous at $x^{*}$
			\end{enumerate}	
			Then the feedback $u^{*}(x)$ defined in \refeq{LIPQP} is unique and Lipschitz continuous with respect to the state at $x^{*}$.
		\end{theorem}	
		We also use the following facts to prove Proposition 1.
		\begin{fact}
			\label{fact1}
			The product of two Lipschitz continuous functions over a bounded set $\X$ is also Lipschitz continuous over $\X$.
		\end{fact}	
		\begin{fact}
			\label{fact2}
			Given a compact set $\X$, a function which is Lipschitz continuous at any point $x\in\X$ is Lipschitz continuous over $\X$.
		\end{fact}	
		\begin{fact}
			\label{fact3}
			Suppose that $b:\X\rightarrow\Real$ is Lipschitz continuous over a set $\X$ and that $\alpha$ is Lipschitz continuous over $\Real$.  Then, the composite function $\alpha\circ b$ is Lipschitz continuous over $\X$.
		\end{fact}	
		We now prove Proposition 1.
		For the barrier-certified policy update (11), inequality constraints represent the affine constraints $\U$ and the barrier certificates, and there are no equality constraints.
		It is, however, possible to augment $u\in \Real^{n_u}$ and consider the QP:
		\begin{align}
			u^{*}_{\rm aug}(x)=&\displaystyle\argmin_{u_{\rm aug}\in\Real^{n_u+1}}u_{\rm aug}^{\T}
			\left[
			\begin{array}{cc}
				H(x)& 0\\ 0 &1
			\end{array}\right]
			u_{\rm aug}+2[v(x)^{\T},1]u_{\rm aug} \label{LIPQP2}\\
			{\rm s.t.}&\;\; \left[A_{(\leq)}(x),\signal{0}\right]u_{\rm aug}\leq [a^{\T}_{(\leq)}(x),1]^{\T}\nn \\
			&\;\; \left[0,0,\ldots,1\right]u_{\rm aug}=0.\nn
		\end{align}	
		Therefore, we can ignore the condition that $A_{(=)}$ has full row rank if there are no equality constraints.
		Because $f$, $g$, $\alpha$, and the derivative of $b$ are Lipschitz continuous over the compact set $\X$, Fact \ref{fact1} and Fact \ref{fact3} imply that
		$\frac{\partial b(x)}{\partial x}g(x)$ and $\frac{\partial b(x)}{\partial x}f(x)+\alpha(b(x))$ are Lipschitz continuous over $\X$.
		Therefore, $A_{\leq}(x)$ and $a_{\leq}(x)$ are Lipschitz continuous over $\X$.
		
		Moreover, because the function $V^{\phi}$ is in the RKHS $\H_V$ associated with the reproducing kernel $\kappa^V(\cdot,\cdot)\in C^{2\x2}(\X\x\X)$, $\frac{\partial V^{\phi}(x)}{\partial x}$ is Lipschitz continuous.
		Therefore, Lipschitz continuity of $g(x)$ and Fact \ref{fact1} imply that $\frac{\partial V^{\phi}(x)}{\partial x}g(x)$ is Lipschitz continuous over the compact set $\X$.
		
		Lastly, $M$ is positive definite and constant over $\X$.
		
		From Theorem \ref{QPsuf} and Fact \ref{fact2}, the policy $\phi^+(x)$ defined in (11) is Lipschitz continuous over $\X$
		if the width of a feasible set $u_w(x)$ is strictly larger than zero at any point in $\X$.
		
		%%%%%%
		\section{DT case and its relation to the existing approaches}
		When the proposed framework is applied to model-based DT-VF approximation in RKHSs, it reproduces some of the existing methods.
		The Bellman equation of a policy $\phi$ for a DT-VF ${\tilde{V}^{\phi}}$ is given by
		\begin{align}
			{\tilde{V}^{\phi}}(x_n)-\gamma \int_{\X}{\tilde{V}^{\phi}}(x_{+})p^{\phi}(x_{+}|x_n)dx_{+}=\tilde{R}^{\phi}(x_n):=\tilde{R}(x_n,\phi(x_n)), \label{belldis}
		\end{align}	
		where $\gamma\in[0,1)$ is the DT discount factor, $x_n$ is the state observed at time instant $n\in\integer_{\geq0}$, $\tilde{R}:\Real^{n_x}\x\U\rightarrow\Real$ is the average immediate cost function at each time instant, and $p^{\phi}(x_{+}|x)$
		is the probability that, given a policy $\phi$, the successor state is $x_{+}$ conditioned on the current state $x$.
		\begin{lemma}[{\cite[page~35~-~Corollary~4]{moor1}}]
			\label{lemmaap}
			Let $\H$ be a Hilbert space associated with the reproducing kernel $\ka{\cdot,\cdot}:\X\x\X\rightarrow\Real$.  If $\X\subset\Real^{n_x}$ is compact and
			$\ka{\cdot,x}$ is continuous for any given $x\in\X$, then $\H$ is the space of continuous functions.
		\end{lemma}	
		\begin{theorem}
			\label{theoRKHSdis}
			Suppose that $\H_{\tilde{V}}$ is an RKHS associated with the reproducing kernel $\kappa^{\tilde{V}}(\cdot,\cdot):\X\x\X\rightarrow\Real$.
			Suppose also that $\X\subset\Real^{n_x}$ is compact, and that $\kappa^{\tilde{V}}(\cdot,x)$ is continuous for any given $x\in\X$.
			Define the operator $\tilde{U}$ as $\tilde{U}(\varphi^{\tilde{V}})(x):=\varphi^{\tilde{V}}(x)-\gamma\int_{\X}\varphi^{\tilde{V}}(x_{+})p^{\phi}(x_{+}|x)dx_{+}$ for any $\varphi^{\tilde{V}}\in\H_{\tilde{V}}$ and for any $x\in\X$, where $\gamma\in[0,1)$.
			Then, the following statements hold.\\
			(a) The space
			\begin{align}
				\H_{\tilde{R}}:=\{\varphi\mid\varphi(x)=\tilde{U}(\varphi^{\tilde{V}})(x),\;\exists\varphi^{\tilde{V}}\in\H_{\tilde{V}},\forall x\in\X\}
			\end{align}
			is an isomorphic Hilbert space of $\H_{\tilde{V}}$ equipped with the inner product defined by
			\begin{align}
				\inpro{\varphi_1,\varphi_2}_{\H_{\tilde{R}}}:=\inpro{\varphi_1^{\tilde{V}},\varphi_2^{\tilde{V}}}_{\H_{\tilde{V}}},\;\varphi_i(x):=\tilde{U}(\varphi_i^{\tilde{V}})(x),\;\forall x\in\X,\;i\in\{1,2\}. \label{inprodis}
			\end{align}
			(b) The Hilbert space $\H_{\tilde{R}}$ has the reproducing kernel given by
			\begin{align}
				\tilde{\kappa}(x,y):=\tilde{U}(\tilde{K}(\cdot,y))(x), \;x,y\in\X,  \label{DTrepro}
			\end{align}
			where
			\begin{align}
				\tilde{K}(\cdot,x)=\kappa^{\tilde{V}}(\cdot,x)-\gamma m^{\kappa}_x.
			\end{align}
			Here, $m^{\kappa}_x\in\H_{\tilde{V}}$ is the embedding satisfying
			\begin{align}
				\inpro{m^{\kappa}_x,\varphi^{\tilde{V}}}_{\H_{\tilde{V}}}=\int_{\X}\varphi^{\tilde{V}}(x_{+})p^{\phi}(x_{+}|x)dx_{+}.
			\end{align}
		\end{theorem}	
		\begin{proof}
			We show that $\tilde{U}$ is bijective linear, and then show that the reproducing kernel in $\H_{\tilde{R}}$ is given by \refeq{DTrepro}.
			\paragraph{Proof of (a)}
			Because the operator $\tilde{U}$ is surjective by definition of $\H_{\tilde{R}}$, we show that $\tilde{U}$ is injective.
			The expectation operator and hence $\tilde{U}$ is linear.
			Hence, it is sufficient to show that ${\rm ker}(\tilde{U})=0$ \cite{linearalgebra}.
			Suppose that $\tilde{U}(\varphi^{\tilde{V}})(x)=0,\;\forall x\in\X$.
			It then follows that $\varphi^{\tilde{V}}(x)=\gamma\int_{\X}\varphi^{\tilde{V}}(x_{+})p^{\phi}(x_{+}|x)dx_{+}$ for all $x\in\X$.
			We show that $\varphi^{\tilde{V}}=0$ by contradiction.
			Since $\X$ is compact and any $\varphi^{\tilde{V}}\in\H_{\tilde{V}}$ is continuous from Lemma \ref{lemmaap}, $\varphi^{\tilde{V}}$ attains the maximum/minimum value at some point $x_{\rm max}$/$x_{\rm min}$, respectively.
			If $\varphi^{\tilde{V}}\neq0$, then we obtain $\varphi^{\tilde{V}}(x_{\rm max})>0$ or $\varphi^{\tilde{V}}(x_{\rm min})<0$.
			If $\varphi^{\tilde{V}}(x_{\rm max})>0$, it follows that
			\begin{align}
				\varphi^{\tilde{V}}(x_{\rm max})=\gamma\int_{\X}\varphi^{\tilde{V}}(x_{+})p^{\phi}(x_{+}|x_{\rm max})dx_{+}<\varphi^{\tilde{V}}(x_{\rm max}),
			\end{align} 
			for $\gamma\in[0,1)$, which is contradictory.
			If $\varphi^{\tilde{V}}(x_{\rm min})<0$, it follows that 
			\begin{align}
				\varphi^{\tilde{V}}(x_{\rm min})=\gamma\int_{\X}\varphi^{\tilde{V}}(x_{+})p^{\phi}(x_{+}|x_{\rm min})dx_{+}>\varphi^{\tilde{V}}(x_{\rm min}),
			\end{align}
			for $\gamma\in[0,1)$, which is also contradictory.
			Hence, $\varphi^{\tilde{V}}=0$, and ${\rm ker}(\tilde{U})=0$.
			Therefore, the correspondence between $\varphi^{\tilde{V}}\in\H_{\tilde{R}}$ and $\varphi=\tilde{U}(\varphi^{\tilde{V}})$ is one-to-one, and inner product preserves by definition \refeq{inprodis}.
			\paragraph{Proof of (b)}
			We show that $\H_{\tilde{R}}$ is an RKHS.
			Because $m^{\kappa}_x\in\H_{\tilde{V}}$ and hence $\tilde{K}(\cdot,x)\in\H_{\tilde{V}}$, it follows that $\tilde{\kappa}(\cdot,x)\in\H_{\tilde{R}}$.
			Moreover, it holds that
			\begin{align}
				\tilde{U}(\varphi^{\tilde{V}})(x)=\inpro{\varphi^{\tilde{V}},\tilde{K}(\cdot,x)}_{\H_{\tilde{V}}},
			\end{align}
			from which we obtain
			\begin{align}
				\inpro{\tilde{\kappa}(\cdot,x),\tilde{\kappa}(\cdot,y)}_{\H_{\tilde{R}}}=\inpro{\tilde{K}(\cdot,x),\tilde{K}(\cdot,y)}_{\H_{\tilde{V}}}
				=\tilde{U}(\tilde{K}(\cdot,y))(x)=\tilde{\kappa}(x,y),
			\end{align}	
			and that
			\begin{align}
				\inpro{\varphi,\tilde{\kappa}(\cdot,x)}_{\H_{\tilde{R}}}=\inpro{\tilde{U}^{-1}(\varphi),\tilde{K}(\cdot,x)}_{\H_{\tilde{V}}}=\tilde{U}(\tilde{U}^{-1}(\varphi))(x)=\varphi(x),\;\forall \varphi\in\H_{\tilde{R}},\forall x\in\X.
			\end{align}	
			Therefore, $\tilde{\kappa}(\cdot,\cdot):\X\x\X\rightarrow\Real$ is the reproducing kernel with which the RKHS $\H_{\tilde{R}}$ is associated.
		\end{proof}	
		\begin{remark}
			If $x_{+}$ is deterministically obtained, $m^{\kappa}_x=\kappa^{\tilde{V}}(\cdot,x_{+})$ for some $x_{+}\in\X$.
		\end{remark}	
		\begin{remark}
			Note that the work in \cite{ohnishi2018safety} considers model-free DT action-value function approximation
			and defines an RKHS over $\mathcal{Z}^2$, where $\mathcal{Z}:=\X\x\U$, while Theorem \ref{theoRKHSdis} is for
			model-based DT-VF approximation.
		\end{remark}
		Provided Theorem \ref{theoRKHSdis}, we show how some of the existing model-based DT-VF approximation methods can be reproduced.
		\paragraph {Online gradient descent on a sequence of Bellman loss functions}
		In \cite{sun2016online}, the residual gradient algorithm is viewed as running online gradient descent on the Bellman loss,
		and the following update equation is given:
		\begin{align}
			\hat{\tilde{V}}^{\phi}_{n+1}=\hat{\tilde{V}}^{\phi}_{n}-\lambda_ne^B_n(\nabla_{\tilde{V}}\hat{\tilde{V}}^{\phi}_{n}(x_n)-\gamma \nabla_{\tilde{V}}\hat{\tilde{V}}^{\phi}_{n}(x_{n+1})), \label{norma}
		\end{align}
		where $e^B_n:=\hat{\tilde{V}}^{\phi}_{n}(x_n)-\gamma \hat{\tilde{V}}^{\phi}_{n}(x_{n+1})-\tilde{R}^{\phi}(x_n)$, and
		$\nabla_{\tilde{V}}\tilde{V}(x)$ is the functional gradient of the
		evaluation functional $\tilde{V}(x)$ at a function $\tilde{V}$.
		Note that it is implicitly assumed that $\tilde{V}$ is in an RKHS $\H_{\tilde{V}}$ and hence $\nabla_{\tilde{V}}\tilde{V}(x)=\kappa^{\tilde{V}}(\cdot,x)$.
		If, on the other hand, we employ $\H_{\tilde{R}}$ defined in Theorem \ref{theoRKHSdis} as the stage of learning,
		and apply online gradient descent, we obtain the update rule:
		\begin{align}
			\hat{\tilde{R}}^{\phi}_{n+1}=\hat{\tilde{R}}^{\phi}_{n}-\lambda_n(\hat{\tilde{R}}^{\phi}_n(x_n)-\tilde{R}^{\phi}(x_n))\tilde{\kappa}(\cdot,x),
		\end{align}
		which results in updating $\tilde{V}^{\phi}_{n}$ by \refeq{norma} within $\H_{\tilde{V}}$.
		\paragraph {MDPs with RKHS embeddings}
		A kernelized version of MDPs has been proposed in \cite{grunewalder2012modelling}.
		Nonparametric nature of kernelized MDPs enables us to handle complicated distributions, high-dimensional data, and continuous states and actions,
		and convergence can be analyzed in the infinite sample case \cite{grunewalder2012modelling}.
		The embedding $m^{\kappa}_x$ might be learned \cite{song2013kernel} or gives information
		about the inaccuracy of a nominal model.
		The model-based VF approximation for MDPs in an RKHS can also be efficiently conducted in our framework.

		\paragraph {Gaussian process temporal difference algorithm}
		To address a probabilistic nature of MDPs, Bayesian approach is a natural option, and GPTD
		was proposed \cite{engel2005reinforcement}.
		We consider a path $(x_n)_{n=0,1,...,N}$ of the state.
		In GPTD, the posterior mean and variance of $\hat{\tilde{V}}^{\phi}$ at a point $x_{*}\in\X$ are given by
		\begin{align}
			&m^{\tilde{V}}(x_{*})={k^{\tilde{V}}_{*}}^{\T}H^{\T}(H\tilde{G}_kH^{\T}+\Sigma)^{-1}\tilde{d}_{N-1},\label{meangptd}\\
			&{\mu^{\tilde{V}}}^2(x_{*})=\kappa^{\tilde{V}}(x_{*},x_{*}) -{k^{\tilde{V}}_{*}}^{\T}H^{\T}(H\tilde{G}_kH^{\T}+\Sigma)^{-1}H{k^{\tilde{V}}_{*}}, \label{covgptd}
		\end{align}	
		where $\tilde{d}_{N-1}\sim\mathcal{N}([\tilde{R}^{\phi}(x_0),\tilde{R}^{\phi}(x_1),\ldots,\tilde{R}^{\phi}(x_{N-1})]^{\T},\Sigma)$ for some $N\in\integer_{\geq0}$,
		$k^{\tilde{V}}_{*}:=[\kappa^{\tilde{V}}(x_{*},x_0),\kappa^{\tilde{V}}(x_{*},x_1),\ldots,\kappa^{\tilde{V}}(x_{*},x_N)]^{\T}$, the $(i,j)$ entry of $\tilde{G}_k\in\Real^{(N+1)\x (N+1)}$ is $\kappa^{\tilde{V}}(x_{i-1},x_{j-1})$, $\Sigma\in\Real^{N\x N}$ is the covariance matrix of $\tilde{d}_{N-1}$, and
		\begin{align}
			H:=\left[
			\begin{array}{ccccc}
				1 & -\gamma & 0 & \ldots & 0\\
				0 & 1 & -\gamma & \ldots & 0\\
				\vdots &  &  &  & \vdots\\
				0 & 0 & \ldots & 1 & -\gamma\\
			\end{array}
			\right]\in\Real^{N\x (N+1)}.
		\end{align}
		If, on the other hand, the RKHS $\H_{\tilde{R}}$ defined in Theorem \ref{theoRKHSdis} is employed as the stage of learning, by letting
		$m^{\kappa}_{x_n}=\kappa^{\tilde{V}}(\cdot, x_{n+1})$ for $n=0,1,...,N-1$, we obtain
		\begin{align}
			&m^{\tilde{V}}(x_{*})={K^{\tilde{V}}_{*}}^{\T}(\tilde{G}+\Sigma)^{-1}\tilde{d}_N,\label{meangpd2}\\
			&{\mu^{\tilde{V}}}^2(x_{*})=\kappa^{\tilde{V}}(x_{*},x_{*}) -{K_{*}^{\tilde{V}}}^{\T}(\tilde{G}+\Sigma)^{-1}K_{*}^{\tilde{V}}, \label{covgpd2}
		\end{align}	
		where $K^{\tilde{V}}_{*}:=[\tilde{K}(x_{*},x_0),\tilde{K}(x_{*},x_1),\ldots,\tilde{K}(x_{*},x_{N-1})]^{\T}$, and the $(i,j)$ entry of $\tilde{G}\in\Real^{N\x N}$ is $\tilde{\kappa}(x_{i-1},x_{j-1})$, which result in the same values as GPTD.
		%%%%
		
		As we have shown, Theorem \ref{theoRKHSdis} unifies
		model-based DT-VF approximation methods working in RKHSs.
		As such, it is able to analyze tracking/convergence etc. straightforwardly by applying already established arguments
		of kernel-based methods.
		The present study enables us to conduct CT-VF approximation in RKHSs, and can also be viewed as a CT version of this result by utilizing a partial derivative reproducing property
		of certain classes of RKHSs.
		%%%%%%
		
		%%%%%%
		\section{Derivative of a Gaussian kernel}
		The function $a(x,y)$ appearing in Section 6 is given by
		\begin{align}
			&a(x,y)=\beta^2-\beta\sum_{i=1}^{n_x}a_1^i(x,y)\{{h^{\phi}}^i(y)-{h^{\phi}}^i(x)\}-\frac{\beta}{2}\sum_{i=1}^{n_x}a_2^i(x,y)\{A^{\phi}_{i,i}(y)+A^{\phi}_{i,i}(x)\}\nn\\
			&+\sum_{i,j=1}^{n_x}a_{1,1}^{j,i}(x,y){h^{\phi}}^i(y){h^{\phi}}^j(x)+\frac{1}{2}\sum_{i,j=1}^{n_x}a_{1,2}^{j,i}(x,y)\{A^{\phi}_{i,i}(y){h^{\phi}}^j(x)-A^{\phi}_{j,j}(x){h^{\phi}}^i(y)\}\nn\\
			&+\frac{1}{4}\sum_{i,j=1}^{n_x}a_{2,2}^{j,i}(x,y)A^{\phi}_{i,i}(y)A^{\phi}_{j,j}(x),
		\end{align}
		where $a_1^i,\;a_2^i,\;a_{1,1}^{j,i},\;a_{1,2}^{j,i}$, and $a_{2,2}^{j,i}$ are defined as
		\begin{align}
			(D^{e_i}\kappa^V)_y(x)=a_1^i(x,y)\kappa^V(x,y):=\frac{-(y^i-x^i)}{\sigma^2}\kappa^V(x,y),
		\end{align}
		%\vspace{-0.5em}
		\begin{align}
			(D^{2e_i}\kappa^V)_y(x)=a_2^i(x,y)\kappa^V(x,y):=\frac{(y^i-x^i)^2-\sigma^2}{\sigma^4}\kappa^V(x,y),
		\end{align}
		%\vspace{-0.5em}
		\begin{align}
			D^{e_j}(D^{e_i}\kappa^V)_y(x)=\begin{cases}
				a_{1,1}^{j,i}(x,y)\kappa^V(x,y):=\frac{(y^i-x^i)(x^j-y^j)}{\sigma^4}\kappa^V(x,y),\hfill i\neq j, \\
				a_{1,1}^{i,i}(x,y)\kappa^V(x,y):=\frac{(y^i-x^i)(x^i-y^i)+\sigma^2}{\sigma^4}\kappa^V(x,y),\hfill i= j,
			\end{cases}
		\end{align}
		%\vspace{-0.5em}
		\begin{align}
			D^{e_j}(D^{2e_i}\kappa^V)_y(x)=\begin{cases}
				a_{1,2}^{j,i}(x,y)\kappa^V(x,y):=\frac{-\{(y^i-x^i)^2-\sigma^2\}(x^j-y^j)}{\sigma^6}\kappa^V(x,y),\hfill i\neq j, \\
				a_{1,2}^{i,i}(x,y)\kappa^V(x,y):=\frac{-\{(y^i-x^i)^2-3\sigma^2\}(x^i-y^i)}{\sigma^6}\kappa^V(x,y),\hfill i= j,
			\end{cases}
		\end{align}
		and
		\begin{align}
			\hspace{-3em}
			D^{2e_j}(D^{2e_i}\kappa^V)_y(x)=\begin{cases}
				a_{2,2}^{j,i}(x,y)\kappa^V(x,y):=\frac{\{(x^i-y^i)^2-\sigma^2\}\{(x^j-y^j)^2-\sigma^2\}}{\sigma^8}\kappa^V(x,y),\hfill i\neq j, \\
				a_{2,2}^{i,i}(x,y)\kappa^V(x,y):=\frac{\{(x^i-y^i)^2-6\sigma^2\}(x^i-y^i)^2+3\sigma^4}{\sigma^8}\kappa^V(x,y),\hfill i= j.
			\end{cases}
		\end{align}	 
		%%%%%%%%%%%
		\section{Derivations of (13)}
		The mean $m(x_{*})$ and the variance $\mu^2(x_{*})$ of $\hat{R}^{\phi}(x_{*})$ at a point $x_{*}\in\X$ are given as
		\begin{align}
			&m(x_{*})=k_{*}^{\T}(G+\mu_o^2I)^{-1}d_N,\\
			&\mu^2(x_{*})=C_K(x_{*},x_{*}):=\kappa(x_{*},x_{*})-k_{*}^{\T}(G+\mu_o^2I)^{-1}k_{*}.
		\end{align}	
		From Appendix \ref{proof}, we know that $U$ is linear and there exists $U^{-1}:\H_{R}\rightarrow\H_V$.
		Then, by following the arguments in {\cite[Equation~(8)]{jidling2017linearly}}, we obtain
		\begin{align}
			&m^V(x_{*})=U^{-1}(m^V)(x_{*})={K^V_{*}}^{\T}(G+\mu_o^2I)^{-1}d_N,\label{mean2}\\
			&{\mu^V}^2(x_{*})=U^{-1}C_K(x_{*},x_{*}){U^{\T}}^{-1}:=U^{-1}{U_{r}}^{-1}C_K(x_{*},x_{*}),
		\end{align}	
		where ${U_{r}}^{-1}$ acts on the second argument of $C_K(x_{*},x_{*})$, e.g.,
		$K(x,y)=U_{r}(\kappa^V(x,\cdot))(y)$.  Therefore, we obtain
		\begin{align}
			{\mu^V}^2(x_{*})&=U^{-1}{U_{r}}^{-1}C_K(x_{*},x_{*})={U_{r}}^{-1}(K(x_{*},\cdot))(x_{*}) -U^{-1}k_{*}^{\T}(G+\mu_o^2I)^{-1}k_{*}{U^{\T}}^{-1}\nn\\
			& =\kappa^V(x_{*},x_{*}) -{K_{*}^V}^{\T}(G+\mu_o^2I)^{-1}K_{*}^V. \label{cov2}
		\end{align}

		\section{Monotone approximation, strong convergence, and sparsity}
		In our proposed framework, the immediate cost function is estimated.  Therefore, if the monotone approximation property and strong convergence are guaranteed
		for the immediate cost function in the RKHS $\H_R$ under certain conditions, these properties are also guaranteed for the VF because of one-to-one correspondence between $\H_R$ and $\H_V$.
		In the RKHS $\H_R$, an estimate of $R^{\phi}$ at time instant $n\in\integer_{\geq0}$ is given by
		$\hat{R}^{\phi}_n(x)=\sum_{i}^{r}c_i\kappa(x,x_i),\;c_i\in\Real,\;r\in\integer_{\geq0}$,
		where $\{x_i\}_{i\in\{1,2,...,r\}}\subset\X$ is the set of samples, and the reproducing kernel $\kappa$ is defined in (9).
		The estimate of the VF $V^{\phi}$ at time instant $n\in\integer_{\geq0}$ is then obtained by
		$\hat{V}^{\phi}_n(x)=U^{-1}(\hat{R}^{\phi}_n)(x)=\sum_{i=1}^{r}c_iK(x,x_i)$, where $K$ is defined in (10).
		Therefore, when an algorithm promoting sparsity is employed and only fewer kernel functions are employed to estimate the immediate cost function, i.e., $r$ is suppressed small, it immediately implies that sparsity is preserved for the estimate of the VF as well.
		
		\section{Experimental settings}
		We present the experimental settings.
		The parameter settings for the Mountain Car problem and the simulated inverted pendulum are summarized in \reftab{sumparam1}, and \reftab{sumparam2}, respectively.
		The parameters were roughly tuned so that the algorithms work reasonably well.
		However, we conducted no elaborative tuning or heuristic approaches, which are crucial to ensure stable improvements of performance in RL, to further improve performances, because the main purpose of the experiments is to show sensitiveness of DT approaches toward the choice of the time interval.
		%%%%%
		%%%%%%%
		\begin{table}[t]
			\caption{Summary of the parameter settings for the Mountain Car problem}
			\label{sumparam1}
			\centering
			\begin{tabular}{lllll}
				\toprule
				Parameters & CTGP & CTKF & GPTD  &DTKF  \\
				\midrule
				Coherence threshold & 0.70 & 0.70 & 0.70 & 0.70\\
				Kernel parameter $\sigma$ & 0.2 & 0.2 & 0.2 & 0.2\\
				Maximum absolute value of control & 1.0  & 1.0 & 1.0 & 1.0 \\
				Control cycle & 1.0 ${\rm (sec)}$ & 1.0 ${\rm (sec)}$ & 1.0 ${\rm (sec)}$ & 1.0 ${\rm (sec)}$ \\
				Discount factor $\beta,\gamma$ & $\beta=0$ & $\beta=0$ & $\gamma=1$ & $\gamma=1$ \\
				Standard deviation $\mu_o$ of the observed cost & 0.1 & $0.1$ & $0.1\Delta t^2$ & $0.1\Delta t^2$ \\
				Cost on controls: Matrix $M$ & $0.001I$ & $0.001I$ & $0.001I$& $0.001I$ \\
				Stochastic term in the SDE: Matrix $A(x,u)$ & $0$ & $0$ & $0$ & $0$ \\
				Step size & -- & 1.8 & -- & 0.4 \\
				\bottomrule
			\end{tabular}
		\end{table}
		\begin{table}[t]
			\caption{Summary of the parameter settings for the simulated inverted pendulum}
			\label{sumparam2}
			\centering
			\begin{tabular}{lll}
				\toprule
				Parameters & CTGP & GPTD  \\
				\midrule
				Coherence threshold & 0.95 & 0.95\\
				Kernel parameter $\sigma$ & 0.2 & 0.2\\
				Maximum absolute value of control & 6  & 6 \\
				Learning time per update & 10 ${\rm (sec)}$ & 10 ${\rm (sec)}$ \\
				Time interval $\Delta t$ & 0.01 ${\rm (sec)}$ & 0.01 ${\rm (sec)}$ \\
				Discount factor $\beta,\gamma$ & $\beta=0.01$ & $\gamma=e^{-0.01*0.01}$ \\
				Standard deviation $\mu_o$ of the observed cost & 0.1 & 0.1 \\
				Cost on controls: Matrix $M$ & $0.1I$ & $0.1I$ \\
				Stochastic term in the SDE: Matrix $A(x,u)$ & $0.01I$ & $0.01I$ \\
				Gravity $g$ & 9.8 & 9.8 \\
				Mass $m$ & 1 & 1 \\
				Length $\ell$ of pendulum & 1 & 1 \\
				Friction effect $\rho$ & 0.01 & 0.01\\
				\bottomrule
			\end{tabular}
		\end{table}
		%%%%%%%%
		%%%%%%%%
	\end{appendices}	
%%%%%

%	
\end{document}